\documentclass[a4paper,11pt,reqno]{amsart}
\usepackage{graphicx}
\usepackage[latin1]{inputenc}
\usepackage{mathrsfs}
\usepackage{dsfont}
\usepackage{hyperref}
\usepackage{amsmath}
\usepackage{amssymb}
\usepackage{amsthm}
\usepackage{amsfonts}
\usepackage{amstext}
\usepackage{amsopn}
\usepackage{amsxtra}
\usepackage{mathrsfs}
\usepackage{dsfont}
\usepackage{esint}
\usepackage{enumitem}

\newtheorem{thm}{Theorem}
\newtheorem{conjecture}{Conjecture}
\newtheorem{lemma}[thm]{Lemma}

\newtheorem{prop}[thm]{Proposition}
\newtheorem{cor}[thm]{Corollary}
\newtheorem{remark}[thm]{Remark}

\newcommand{\R}{\mathbb{R}}
\newcommand{\Z}{\mathbb{Z}}
\newcommand{\C}{\mathbb{C}}

\newcommand{\dps}{\displaystyle}
\newcommand{\ii}{\infty}
\newcommand\1{{\ensuremath {\mathds 1} }}
\renewcommand\phi{\varphi}

\newcommand{\wto}{\rightharpoonup}

\newcommand{\cV}{\mathcal{V}}

\newcommand{\cK}{\mathcal{K}}

\newcommand{\cD}{\mathcal{D}}

\newcommand{\alp}{\boldsymbol{\alpha}}
\newcommand\pscal[1]{{\ensuremath{\left\langle #1 \right\rangle}}}
\newcommand{\norm}[1]{ \left\| #1 \right\|}

\renewcommand{\geq}{\geqslant}
\renewcommand{\leq}{\leqslant}

\renewcommand{\tilde}{\widetilde}

\newcommand{\be}{\begin{equation}}
\newcommand{\ee}{\end{equation}}
\newcommand{\bq}{\begin{equation}}
\newcommand{\eq}{\end{equation}}

\newcommand{\un}{{1\!\!1}_4}

\newcommand{\Frac}{\displaystyle \frac}

\newcommand{\Sup}{\displaystyle \sup}

\newcommand{\eps}{\varepsilon}
\usepackage{color}

\newcommand{\nn}{\nonumber}

\newcommand{\Spec}{{\rm Sp}}

\title[Dirac-Coulomb operators with general charge distribution II]{Dirac-Coulomb operators with general charge distribution\\[0.2cm] II. The lowest eigenvalue}

\author[M.J. Esteban]{Maria J. Esteban}
\address{CEREMADE, CNRS, Université Paris-Dauphine, PSL Research University, Place de Lattre de Tassigny, 75016 Paris, France} 
\email{esteban@ceremade.dauphine.fr}

\author[M. Lewin]{Mathieu Lewin}
\address{CEREMADE, CNRS, Université Paris-Dauphine, PSL Research University, Place de Lattre de Tassigny, 75016 Paris, France} 
\email{mathieu.lewin@math.cnrs.fr}

\author[\'E. Séré]{\'Eric Séré}
\address{CEREMADE, Université Paris-Dauphine, PSL Research University, CNRS, Place de Lattre de Tassigny, 75016 Paris, France} 
\email{sere@ceremade.dauphine.fr}

\date{\today}

\DeclareRobustCommand{\SkipTocEntry}[5]{}

\begin{document}

\begin{abstract}
Consider the Coulomb potential $-\mu\ast|x|^{-1}$ generated by a non-negative finite measure $\mu$. It is well known that the lowest eigenvalue of the corresponding Schrödinger operator $-\Delta/2-\mu\ast|x|^{-1}$ is minimized, at fixed mass $\mu(\R^3)=\nu$, when $\mu$ is proportional to a delta. In this paper we investigate the conjecture that the same holds for the Dirac operator $-i\alp\cdot\nabla+\beta-\mu\ast|x|^{-1}$. In a previous work on the subject we proved that this operator is self-adjoint when $\mu$ has no atom of mass larger than or equal to 1, and that its eigenvalues are given by min-max formulas. Here we consider the critical mass $\nu_1$, below which the lowest eigenvalue does not dive into the lower continuum spectrum for all $\mu\geq0$ with $\mu(\R^3)<\nu_1$. We first show that $\nu_1$ is related to the best constant in a new scaling-invariant Hardy-type inequality. Our main result is that for all $0\leq\nu<\nu_1$, there exists an optimal measure $\mu\geq0$ giving the lowest possible eigenvalue at fixed mass $\mu(\R^3)=\nu$, which concentrates on a compact set of Lebesgue measure zero. The last property is shown using a new unique continuation principle for Dirac operators. The existence proof is based on the concentration-compactness principle. 

\medskip

\noindent\textbf{Keywords:} Dirac operators, min-max formulas, variational methods, eigenvalue optimization, spectral theory, concentration-compactness principle, unique continuation principle
\end{abstract}

\maketitle

\tableofcontents

\section{Introduction}

Consider a non-negative finite Borel measure $\mu$ on $\R^3$ and the corresponding linear Schrödinger operator
$$-\frac{\Delta}{2}-\mu\ast\frac{1}{|x|},$$
which describes a non-relativistic electron moving in the Coulomb potential generated by the charge distribution $\mu$. The lowest eigenvalue of this operator is given by the variational principle
\begin{multline}
\lambda_1\left(-\frac{\Delta}{2}-\mu\ast\frac{1}{|x|}\right)\\=\inf_{\substack{\phi\in H^1(\R^3)\\ \int_{\R^3}|\phi|^2=1}}\left\{\frac12\int_{\R^3}|\nabla\phi(x)|^2\,dx-\int_{\R^3}\left(\mu\ast\frac1{|\cdot|}\right)(x)|\phi(x)|^2\,dx\right\} 
\label{eq:Schrodinger_minimum}
\end{multline}
and from this we deduce immediately that it is a \emph{concave} function of $\mu$. Therefore, it is minimized, at fixed mass $\mu(\R^3)=\nu$, when $\mu$ is proportional to a delta and we have 
\begin{equation}
 \lambda_1\left(-\frac{\Delta}{2}-\mu\ast\frac{1}{|x|}\right)\geq \lambda_1\left(-\frac{\Delta}{2}-\frac{\mu(\R^3)}{|x|}\right)=-\frac{\mu(\R^3)^2}{2}
 \label{eq:estim_Schrodinger}
\end{equation}
for every $\mu\geq0$. The interpretation is that the lowest possible electronic energy is reached by taking the most concentrated charge distribution, at fixed total charge $\mu(\R^3)$. This result was generalized to many-body systems in~\cite{LieSim-78,HofOst-80,Lieb-82}. 

In the presence of a large charge distribution $\mu$, for instance generated by a heavy nucleus, the electron will naturally attain high velocities, of the order of the speed of light. Relativistic effects then become important and a proper description should involve the Dirac operator~\cite{Thaller,EstLewSer-08}. The above Schrödinger operator is then replaced by 
\begin{equation}
D_0-\mu\ast\frac1{|x|}=-i\sum_{j=1}^3\alpha_j\partial_{x_j}+\beta-\mu\ast\frac1{|x|}
\label{eq:general_form}
\end{equation}
where $\alpha_1,\alpha_2,\alpha_3,\beta$ are the Dirac matrices recalled below in~\eqref{eq:Dirac_matrices}. One important difference with the Schrödinger case is that the spectrum of $D_0$ is $\R\setminus(-1,1)$, hence is unbounded both from below and above. This is also the essential spectrum of $D_0-\mu\ast|x|^{-1}$. The eigenvalues in the gap $(-1,1)$ physically correspond to stationary states of the relativistic electrons. Therefore it seems natural to expect that the lowest eigenvalue in $(-1,1)$ will again be minimized for the Dirac measure $\mu(\R^3)\delta_0$, like in the Schrödinger case~\eqref{eq:estim_Schrodinger}. We would then have 
\begin{equation}
\boxed{\lambda_1\left(D_0-\mu\ast\frac1{|x|}\right)\overset{?}\geq\sqrt{1-\mu(\R^3)^2}}
\label{eq:conjecture_eigenvalue_intro}
\end{equation}
for every positive $\mu$ with $\mu(\R^3)\leq1$. The goal of this paper is to investigate Conjecture~\eqref{eq:conjecture_eigenvalue_intro}. When $\mu$ is the sum of two deltas, the conjecture is supported by numerical simulations~\cite{ArtSurIndPluSto-10,McConnell-13} and it was already mentioned by Klaus in~\cite[p.~478]{Klaus-80b} and by Briet-Hogreve in~\cite[Sec.~2.4]{BriHog-03}. To our knowledge, this is the first time that the problem is investigated for a general measure $\mu$. 

\medskip

Proving the inequality~\eqref{eq:conjecture_eigenvalue_intro} is much more difficult than for~\eqref{eq:estim_Schrodinger}. The first technical difficulty is that we have to clarify which operator we are considering. For a singular measure $\mu$ the operator $D_0-\mu\ast|x|^{-1}$ can have \emph{several self-adjoint extensions}, all with a different point spectrum. Even in the simple case $\mu=\nu\delta_0$, the Dirac-Coulomb operator $D_0-\nu|x|^{-1}$ has infinitely many self-adjoint extensions when $\nu>\sqrt{3}/2$~\cite{Thaller}. This is due to the fact that the Coulomb potential has a critical scaling with regard to the order-one differential operator $D_0$. This problem does not arise for the Schrödinger operator $-\Delta/2-\mu\ast|x|^{-1}$ which is essentially self-adjoint for every finite measure~$\mu$, by Hardy's inequality. 

The solution to this problem has been found many years ago (see~\cite{EstLewSer-19} for a complete list of references). It was proved that for $\nu<1$ the operator with truncated potential $D_0-\nu|x|^{-1}\1(|x|\geq r)$ converges in the norm-resolvent sense to a unique \emph{`distinguished'} self-adjoint extension. This extension is also the unique one for which the domain of the operator is included in the energy space $H^{1/2}(\R^3,\C^4)$. The lowest eigenvalue of this extension is equal to $\sqrt{1-\nu^2}$, which appears on the right side of~\eqref{eq:conjecture_eigenvalue_intro}. A similar (though more complicated) result holds at $\nu=1$~\cite{EstLewSer-19}. 

In a previous paper~\cite{EstLewSer-20a_ppt} we extended this result to an arbitrary measure~$\mu$ and proved the existence of a similar distinguished self-adjoint extension for $D_0-\mu\ast|x|^{-1}$ with domain in $H^{1/2}(\R^3,\C^4)$, under the sole assumptions that 
$$|\mu|(\R^3)<\ii\qquad\text{ and }\qquad |\mu(\{R\})|<1\qquad \text{for all $R\in\R^3$.}$$ 
In other words, no atom with a weight larger or equal than one is allowed in the measure $\mu$. We also gave in~\cite{EstLewSer-20a_ppt} a characterization of the domain using a method introduced in~\cite{EstLos-07, EstLos-08} and recently revisited in~\cite{EstLewSer-19,SchSolTok-20}. In this paper, we will always work with the so-defined distinguished extension. No prior knowledge of the technical construction in~\cite{EstLewSer-20a_ppt} will however be needed to follow our arguments. 

Next, we need to address what it means to be the `lowest' eigenvalue of an operator in a gap of its spectrum. If we look at the operator $D_0-t\mu\ast|x|^{-1}$ with $\mu\geq0$, this eigenvalue is defined unambiguously for small $t$, since the point spectrum is known to be located in a neighborhood of $1$, in the upper part of the spectral gap. It might however happen that when we increase $t$ this lowest eigenvalue decreases until it reaches $-1$ and finally dissolves in the lower essential spectrum. The second eigenvalue then suddenly becomes the lowest one in the gap. To overcome this difficulty and ensure a reasonable definition, we use a continuation argument and ask that for all $t\in[0,1]$ the lowest eigenvalue in $(-1,1)$ of $D_0-t\mu\ast|x|^{-1}$ never approaches $-1$. This naturally leads us to determine the critical mass $\nu_1$ for which we have $\lambda_1(D_0-\mu\ast|x|^{-1})>-1$ for all $\mu(\R^3)<\nu_1$, so that our conjecture~\eqref{eq:conjecture_eigenvalue_intro} makes sense at least for $\mu(\R^3)<\nu_1$. We will prove in Theorem~\ref{thm:critical_nu's} below that this critical number is also the \textbf{best constant in the Hardy-type inequality}
\begin{equation}
\boxed{  \int_{\R^3}\frac{|\sigma\cdot\nabla\phi|^2}{\mu\ast|x|^{-1}}\,dx\geq\frac{\nu_1^2}{\mu(\R^3)^2}\int_{\R^3}\left(\mu\ast\frac1{|x|}\right)|\phi|^2\,dx} 
 \label{eq:Hardy_intro}
\end{equation}
for every $\phi\in C^\ii_c(\R^3,\C^2)$ and every finite non-negative measure $\mu\geq0$, where $\sigma_1,\sigma_2,\sigma_3$ are here the $2\times2$ Pauli matrices recalled later in~\eqref{eq:Pauli_matrices}. To our knowledge, inequalities of the type~\eqref{eq:Hardy_intro} with two unknowns $(\phi,\mu)$ have never been studied in the literature. Using an inequality of Tix~\cite{Tix-98} we will see that 
$$0.9\simeq\frac2{\pi/2+2/\pi}\leq \nu_1\leq1.$$

The third difficulty is that the lowest Dirac eigenvalue in the gap $(-1,1)$ is not given by a minimum like in~\eqref{eq:Schrodinger_minimum}. In fact it is given by a min-max formula~\cite{GriSie-99, DolEstSer-00, SchSolTok-20,EstLewSer-20a_ppt}. Although this variational characterization can be used to show that $\lambda_1(D_0-\mu\ast|x|^{-1})$ is monotone in $\mu$, it is not obviously concave and this prevents us from using a simple argument such as~\eqref{eq:estim_Schrodinger}. The monotonicity implies already that Conjecture~\eqref{eq:conjecture_eigenvalue_intro} holds for a radially symmetric measure~$\mu$, because we then have the \emph{pointwise} bound
$$\left(\mu\ast\frac1{|\cdot|}\right)(x)\leq \frac{\mu(\R^3)}{|x|},$$
by Newton's theorem~\cite{LieLos-01}. Therefore, the problem reduces to knowing whether the radial symmetry can be broken.\footnote{In the same spirit, think of the second Dirichlet eigenvalue of a domain $\Omega\subset\R^d$ which is also given by a min-max formula. It is minimized at fixed $|\Omega|$ by two disjoint balls (Hong-Krahn-Szegö inequality) contrary to the first eigenvalue, which is minimized by a unique ball (Faber-Krahn inequality)~\cite{Henrot-06}.} 

\medskip

To make some progress on Conjecture~\eqref{eq:conjecture_eigenvalue_intro}, we investigate in this paper the detailed properties of the lowest possible eigenvalue among all possible measures $\mu$ with a fixed maximal charge $\nu<\nu_1$. That is, we look at the minimization problem
\begin{equation}
 \boxed{\lambda_1(\nu):=\inf_{\substack{\mu\geq0\\ \mu(\R^3)\leq\nu}}\lambda_1\left(D_0-\mu\ast\frac1{|x|}\right).}
 \label{eq:def_lambda_1_intro}
\end{equation}
In this language, our conjecture~\eqref{eq:conjecture_eigenvalue_intro} means that $\nu_1=1$ and $\lambda_1(\nu)=\sqrt{1-\nu^2}$. Our main result is Theorem~\ref{thm:existence_optimal_measure} below, which states that for any $0\leq \nu<\nu_1$ there exists at least \textbf{one minimizing measure for $\lambda_1(\nu)$} and that any such minimizer must necessarily \textbf{concentrate on a compact set of Lebesgue measure zero}. The last property is shown by using a unique continuation principle for Dirac operators which we have not been able to locate in the literature and which is the object of Appendix~\ref{app:unique_continuation}. Thus, although we are not yet able to prove the conjecture, we can already ascertain that an optimal measure ought to be singular. Since we have no other information on this minimizer, this justifies the need to study Dirac operators with the Coulomb potential generated by any singular measure. 

The existence of the optimal measure is proved by a rather delicate adaptation of techniques from nonlinear analysis to the context of Dirac operators. The first eigenvalue is a highly nonlinear function of the measure $\mu$, even if the operator only depends linearly on $\mu$. Our main enemy here is the action of the non-compact group of space translations and this will be controlled using Lions' concentration-compactness method~\cite{Lions-84,Lions-84b,Lions-85a,Lions-85b,Struwe,Lewin-conc-comp}. The main difficulty will be to prove that the problem is \textbf{locally compact} under the assumption that $0\leq\nu<\nu_1$ and this is another reason why the critical mass $\nu_1$ plays a central role in this work. In spirit, the local compactness is because the eigenvalue cannot dive into the lower continuous spectrum by definition of $\nu_1$. But the actual proof is rather involved  and deeply relies on variational arguments using the min-max characterization of the first eigenvalue from our previous work~\cite{EstLewSer-20a_ppt}.

The paper is organized as follows. Our main results are all stated in the next section, whereas Sections~\ref{sec:proof_critical_nu's} and~\ref{sec:proof_existence} contain the proofs of our main theorems. Except for the proof of two technical results (Lemmas~\ref{lem:essential_spectrum} and~\ref{lem:CV_K_R_eta_eps} below), this article does not require a detailed knowledge of our previous work~\cite{EstLewSer-20a_ppt}. We provide in Section~\ref{sec:rappels} a summary of the main results of~\cite{EstLewSer-20a_ppt}, which is sufficient to follow the arguments of this paper. 

\bigskip

\noindent{\textbf{Acknowledgement.}} This project has received funding from the European Research Council (ERC) under the European Union's Horizon 2020 research and innovation programme (grant agreement MDFT No 725528 of M.L.), and from the Agence Nationale de la Recherche (grant agreement molQED).

\section{Main results}\label{sec:results}

\subsection{Dirac operators with a general charge distribution}\label{sec:rappels}

In this section we recall several results from~\cite{EstLewSer-20a_ppt} concerning Dirac operators of the form~\eqref{eq:general_form}. We work in a system of units for which $m=c=\hbar=1$. The free Dirac operator $D_0$ is given by
\begin{equation}
D_0\ = -i\boldsymbol{\alpha}\cdot\boldsymbol{\nabla} + \beta = \ - i\sum^3_{k=1} \alpha_k \partial_{x_k} \ + \ {\bf \beta},
\label{eq:def_Dirac}
\end{equation}
where $\alpha_1$, $\alpha_2$, $\alpha_3$ and $\beta$ are Hermitian matrices which  satisfy the following anticommutation relations:
\begin{equation*} \label{CAR}
\left\lbrace
\begin{array}{rcl}
 {\alpha}_k
{\alpha}_\ell + {\alpha}_\ell
{\alpha}_k  & = &  2\,\delta_{k\ell}\,\un,\\
 {\alpha}_k {\beta} + {\beta} {\alpha}_k
& = & 0,\\
\beta^2 & = & \un.
\end{array} \right. 
\end{equation*}
The usual representation in $2\times 2$ blocks is given by 
\begin{equation}
  \beta=\left( \begin{matrix} I_2 & 0 \\ 0 & -I_2 \\ \end{matrix} \right),\quad \; \alpha_k=\left( \begin{matrix}
0 &\sigma_k \\ \sigma_k &0 \\ \end{matrix}\right)  \qquad (k=1, 2, 3)\,,
 \label{eq:Dirac_matrices}
\end{equation}
where the Pauli matrices are defined as
\begin{equation}
\sigma _1=\left( \begin{matrix} 0 & 1
\\ 1 & 0 \\ \end{matrix} \right),\quad  \sigma_2=\left( \begin{matrix} 0 & -i \\
i & 0 \\  \end{matrix}\right),\quad  \sigma_3=\left( 
\begin{matrix} 1 & 0\\  0 &-1\\  \end{matrix}\right) \, .
\label{eq:Pauli_matrices}
\end{equation}
The operator $D_0$ is self-adjoint on $H^1(\R^3,\C^4)$ and its spectrum is $\Spec(D_0)=(-\ii,-1]\cup[1,\ii)$~\cite{Thaller}. 

In~\cite[Thm.~1]{EstLewSer-20a_ppt}, we have shown that under the sole condition that 
\begin{equation}
|\mu|(\R^3)<\ii,\qquad |\mu(\{R\})|<1\qquad \text{for all $R\in\R^3$,} 
 \label{eq:hyp_mu}
\end{equation}
the operator $D_0-\mu\ast|x|^{-1}$  has a \emph{unique self-adjoint extension} on $H^1(\R^3,\C^4)$, whose domain is included in $H^{1/2}(\R^3,\C^4)$. The domain of the extension satisfies
\begin{equation}
 \cD\left(D_0-\mu\ast\frac{1}{|x|}\right)\subset H^1\left(\R^3\setminus\bigcup_{j=1}^KB_r(R_j)\right)
 \label{eq:domain_in_H1_loc}
\end{equation}
for all $r>0$, where $R_1,...,R_K\in\R^3$ are all the points so that $|\mu(\{R_j\})|\geq1/2$. In addition, this operator is the norm-resolvent limit of the similar operators with a truncated Coulomb potential. This result is an extension of several older works on the subject (reviewed for instance in~\cite[Sec.~1.3]{EstLewSer-19}), which have however been mainly concerned with potentials pointwise bounded by a pure Coulomb potential $\nu/|x|$. The case of a finite sum of deltas was considered by Nenciu \cite{Nenciu-77} and Klaus \cite{Klaus-80b}.

In~\cite[Thm.~2--3]{EstLewSer-20a_ppt} we have then considered the particular case when $\mu\geq0$ and we have proved that the domain of the operator can be written in the form
\begin{multline}
\cD(D_0-V_\mu)=\bigg\{\Psi=\begin{pmatrix}\phi\\\chi\end{pmatrix}\in L^2(\R^3,\C^4)\  :\\ \phi\in\cV_\mu,\quad  D_0\Psi-V_\mu\Psi\in L^2(\R^3,\C^4)\bigg\}
\label{eq:inclusion_domain_cV}
\end{multline}
with the Sobolev-type space with weight $(1+V_\mu)^{-1}$
\begin{equation}
\cV_\mu=\left\{\phi\in L^2(\R^3,\C^2)\ :\ \exists g\in L^2(\R^3,\C^2),\quad \sigma\cdot \nabla\phi=(1+V_\mu)^{1/2}g\right\}
\label{eq:cV_mu_maximal}
\end{equation}
endowed with the norm
\begin{equation}
 \norm{\phi}_{\cV_\mu}^2=\int_{\R^3}|\phi(x)|^2\,dx+\int_{\R^3}\frac{|\sigma\cdot \nabla\phi(x)|^2}{1+V_\mu(x)}\,dx
 \label{eq:norm_V}
\end{equation}
Here we have defined for shortness the Coulomb potential
$$\boxed{V_\mu:=\mu\ast\frac{1}{|x|}.}$$
In~\eqref{eq:inclusion_domain_cV} and~\eqref{eq:cV_mu_maximal}, $D_0\Psi$, $V_\mu\Psi$ and $\sigma\cdot \nabla\phi$ are understood in the sense of distributions. That one can characterize the domain of the distinguished self-adjoint extension in terms of a Sobolev space with weight for the upper spinor $\phi$ was realized first in~\cite{EstLos-07,EstLos-08} and later revisited in~\cite{EstLewSer-19,SchSolTok-20}. 

For the convenience of the reader, let us quickly recall how the two-component space $\cV_\mu$ naturally arises in the four-component Dirac problem. It is convenient to consider the eigenvalue equation 
$$(D_0-V_\mu)\,\Psi=\lambda\Psi,\qquad \Psi=\begin{pmatrix}\phi\\\chi\end{pmatrix}$$
for the four-component wavefunction $\Psi$ and to rewrite it in terms of the upper and lower components $\phi$ and $\chi$ as
$$\begin{cases}
(1-V_\mu)\phi-i\sigma\cdot \nabla \chi = \lambda\phi,\\
-(1+V_\mu)\chi-i\sigma\cdot \nabla \phi = \lambda\chi.
  \end{cases}$$
Solving the equation for $\chi$ and inserting it in the equation of $\phi$, one finds (at least formally) that 
\begin{equation}
\left(-\sigma\cdot \nabla\frac{1}{1+\lambda+V_\mu}\sigma\cdot \nabla+1-\lambda-V_\mu\right)\phi=0.
\label{eq:nonlinear_eigenvalue_equation}
\end{equation}
With these manipulations, we have therefore transformed the (strongly indefinite) Dirac eigenvalue problem for $\Psi$ into an elliptic eigenvalue problem, nonlinear in the parameter $\lambda$, for the upper spinor $\phi$. This is reminiscent of the Schur complement formula and the Feshbach-Schur method (see, e.g.,~\cite[Sec.~2.3.1]{AnaLew-20}), and has some interesting numerical advantages for Dirac operators~\cite{DolEstSerVan-00,DesDolEstIndSer-03,KulKolRut-04,ZhaKulKol-04}. The formal operator on the left of~\eqref{eq:nonlinear_eigenvalue_equation} is associated with the quadratic form
\begin{equation}
 q_\lambda(\phi):=\int_{\R^3}\frac{|\sigma\cdot\nabla\phi|^2}{1+\lambda+V_\mu}\,dx+\int_{\R^3}(1-\lambda-V_\mu)|\phi|^2
 \label{eq:def_q_lambda}
\end{equation}
and we proved in~\cite{EstLewSer-20a_ppt} that this quadratic form is bounded from below, and equivalent to the $\cV_\mu$--norm introduced in~\eqref{eq:norm_V} for all $\lambda>-1$, under the sole condition that $\mu\geq0$ and~\eqref{eq:hyp_mu} holds. This allows one to give a meaning to the operator in~\eqref{eq:nonlinear_eigenvalue_equation} by means of the Riesz-Friedrichs method. This is how the space $\cV_\mu$ naturally arises for the upper component $\phi$.  

It turns out that there is a variational interpretation of~\eqref{eq:nonlinear_eigenvalue_equation} in the four-spinor space, which is related to the fact that the energy $\pscal{\Psi,(D_0-V_\mu)\Psi}$ (interpreted here in the quadratic form sense in $H^{1/2}(\R^3,\C^4)$) is essentially concave in $\chi$ and convex in $\phi$. More precisely, let us define as in~\cite{EstLewSer-20a_ppt} the min-max values
\begin{equation}
\lambda_k(D_0-V_\mu) := \  \inf_{
 \scriptstyle W \ {\rm subspace \ of \ } F^+  \atop  \scriptstyle {\rm dim}
\ W =
k  } \  \Sup_{  \scriptstyle \Psi \in ( W \oplus F^- ) \setminus \{ 0 \} } \
\Frac{\pscal{\Psi, (D_0-V_\mu)\,\Psi}}{\|\Psi\|^2}
\label{eq:min-max}  
\end{equation}
for $k \geq 1$, where $F$ is any chosen vector space satisfying 
$$\label{telesc}  C^\ii_c(\R^3,\C^4)\subseteq F \subseteq H^{1/2}(\R^3,\C^4)\,,$$  
and 
$$F^+:=\left\{\Psi=\begin{pmatrix}\phi\\0\\\end{pmatrix}\in F\right\},\qquad F^-:=\left\{\Psi=\begin{pmatrix}0\\\chi\\\end{pmatrix}\in F\right\}.$$
Then it is proved in~\cite{EstLewSer-20a_ppt} that  for $\mu\geq0$ satisfying~\eqref{eq:hyp_mu} and $\mu\neq0$,
\begin{itemize}
 \item[(i)] $\lambda_k(D_0-V_\mu)$ is independent of the chosen space $F$;
 \item[(ii)] $\lambda_k(D_0-V_\mu)\in[-1,1)$ for all $k$;
 \item[(iii)] $\lim_{k\to\ii}\lambda_k(D_0-V_\mu)=1$;
 \item[(iv)] If $k_0$ is the first integer so that $\lambda_{k_0}(D_0-V_\mu)>-1$, then $(\lambda_k(D_0-V_\mu))_{k\geq k_0}$ are all the eigenvalues of $D_0-V_\mu$ in the interval $(-1,1)$, arranged in non-decreasing order and repeated in case of multiplicity.
\end{itemize}
This extends to the case of general Coulomb potentials several previous results about min-max formulas for eigenvalues of Dirac operators~\cite{GriSie-99, DolEstSer-00, DolEstSer-00a, DolEstSer-03, DolEstSer-06,MorMul-15,Muller-16,EstLewSer-19,SchSolTok-20}. 

For $k=1$, one finds after solving the maximization over $\chi$ in~\eqref{eq:min-max} that $\lambda_1(D_0-V_\mu)$ is in fact the \emph{lowest $\lambda\in[-1,1]$} so that 
$$0\in\Spec\left(-\sigma\cdot \nabla\frac{1}{1+\lambda+V_\mu}\sigma\cdot \nabla+1-\lambda-V_\mu\right),$$
the operator appearing in~\eqref{eq:nonlinear_eigenvalue_equation}. Should $\lambda_1(D_0-V_\mu)$ be above $-1$, then it is the first eigenvalue of $D_0-V_\mu$ and $\phi$ is the \emph{first eigenvector} of the operator in~\eqref{eq:nonlinear_eigenvalue_equation}. 

Finally, we recall the Birman-Schwinger principle~\cite{Nenciu-76,Klaus-80b} which states that $\lambda\in(-1,1)$ is an eigenvalue of $D_0-V_\mu$ if and only if $1$ is an eigenvalue of the bounded self-adjoint operator 
\begin{equation}
K_\lambda=\sqrt{V_\mu}\frac{1}{D_0-\lambda}\sqrt{V_\mu}. 
\label{eq:def_K_V}
\end{equation}
The spectrum of $K_\lambda$ is increasing with $\lambda$. This operator was studied at length by Klaus~\cite{Klaus-80b} for $\mu$ the sum of two deltas and in~\cite{EstLewSer-20a_ppt} for general measures. 

In this paper, we are particularly interested in understanding the first min-max level $\lambda_1(D_0-V_\mu)$. Our first goal will be to determine under which condition on the mass of $\mu$ this number is larger than $-1$, hence is the first eigenvalue. 

\subsection{Definition of two critical coupling constants $\nu_0$ and $\nu_1$}

Let us consider any non-negative finite measure $\mu\neq0$ and call 
$$\nu_{\rm max}(\mu):=\max_{R\in\R^3}\mu(\{R\})\in[0,\ii)$$
its largest atom. As we have recalled in the previous section, the operator
$$D_0-t\mu\ast\frac1{|x|}$$
has a distinguished self-adjoint extension for all $0\leq t<\nu_{\rm max}(\mu)^{-1}$, by~\cite[Thm.~1]{EstLewSer-20a_ppt}. The min-max formula~\eqref{eq:min-max} and the Birman-Schwinger principle hold.
Next, we consider the ray $\{t\mu\}_{t>0}$ and ask ourselves at which mass $t\mu(\R^3)$ the first eigenvalue crosses 0 or approaches the bottom $-1$ of the spectral gap. We therefore look at the first min-max level as in~\eqref{eq:min-max}
$$\lambda_1(D_0-tV_\mu):=\inf_{\phi\in H^{1/2}(\R^3,\C^2)\setminus\{0\}}\sup_{\chi\in H^{1/2}(\R^3,\C^2)}\frac{\pscal{\begin{pmatrix}\phi\\\chi\end{pmatrix},\left(D_0-tV_\mu\right)\begin{pmatrix}\phi\\\chi\end{pmatrix}
}}{\|\phi\|^2+\|\chi\|^2}$$
with the potential $V_\mu=\mu\ast|x|^{-1}$, for all $0\leq t<\nu_{\rm max}(\mu)^{-1}$.  Then $t\mapsto \lambda_1(D_0-tV_{\mu})$ is a non-increasing continuous function of $t$ which is the first eigenvalue of $D_0-t V_\mu$ as soon as it stays above $-1$, by~\cite[Thm.~4]{EstLewSer-20a_ppt}. In the limit $t\to0$ we have
$$\lim_{t\to0^+}\lambda_1(D_0-tV_{\mu})=1,$$
that is, for small $t$ the first eigenvalue emerges from $+1$. We have two choices here. Either the eigenvalue decreases and approaches the bottom of the gap $-1$ at some critical $t<\nu_{\rm max}(\mu)^{-1}$, or it stays above it in the whole interval $(0,\nu_{\rm max}(\mu)^{-1})$. We call
$$\nu_1(\mu):=\mu(\R^3)\;\sup\big\{t<\nu_{\rm max}(\mu)^{-1}\ :\ \lambda_1(D_0-V_{t\mu})>-1\big\}$$
the corresponding critical mass. Similarly, we may define
$$\nu_0(\mu):=\mu(\R^3)\;\sup\big\{t<\nu_{\rm max}(\mu)^{-1}\ :\ \lambda_1(D_0-V_{t\mu})>0\big\}.$$
This is the unique value of $t\mu(\R^3)$ for which the first eigenvalue is equal to $0$ if it exists (otherwise it is taken equal to $\mu(\R^3)/\nu_{\rm max}(\mu)$). Of course we could look at a similarly-defined critical number $\nu_\lambda(\mu)$ for all $\lambda\in(-1,1)$ but we concentrate on the cases $\lambda\in\{0,1\}$ since we think they play a special role. By definition we have $\nu_0(\mu)\leq \nu_1(\mu)$. 

By continuity and monotonicity we have $\lambda_1(D_0-tV_\mu)>0$ for all $0\leq t\mu(\R^3)<\nu_0(\mu)$ and $\lambda_1(D_0-tV_\mu)>-1$ for all $0<t\mu(\R^3)<\nu_1(\mu)$. As an example, in the pure Coulomb case $\mu=\delta_0$ we have 
$$\nu_{\rm max}(\delta_0)=\nu_0(\delta_0)=\nu_1(\delta_0)=1.$$
The first eigenvalue reaches $0$ but it never approaches $-1$. 

Note that our definitions are invariant if we multiply the measure $\mu$ by any positive number:
$$\nu_{\rm max}(t\mu)=t\,\nu_{\rm max}(\mu),\qquad \nu_0(t\mu)=\nu_0(\mu),\qquad \nu_1(t\mu)=\nu_1(\mu).$$
When discussing $\nu_0(\mu)$ and $\nu_1(\mu)$ it will often be convenient to take $\mu$ a probability measure. But when looking at the first eigenvalue, our measures $\mu$ will be assumed to satisfy the condition $\mu(\R^3)\leq\nu$. We hope this does not create any confusion. 

In this paper we are interested in the following minimization problems:
\begin{equation}
\boxed{\nu_0:=\inf_{\substack{\mu\geq0\\ \mu\neq0}}\nu_0(\mu),\qquad \nu_1:=\inf_{\substack{\mu\geq0\\ \mu\neq0}}\nu_1(\mu)}
\label{eq:def_nu_c}
\end{equation}
which are respectively the smallest charge for which an eigenvalue can approach $0$ or $-1$, for some probability measure $\mu$. For $\nu<\nu_1$ we also study the minimization problem
\begin{equation}
\boxed{\lambda_1(\nu)=\inf_{\substack{\mu\geq0\\ \mu(\R^3)\leq\nu}}\lambda_1\left(D_0-\mu\ast\frac1{|x|}\right).}
\label{eq:min_lambda_1}
\end{equation}
Since $\nu<\nu_1$, then we know that the eigenvalue in the infimum is always greater than $-1$. As we will see later in Theorem~\ref{thm:existence_optimal_measure}, it turns out that we also have $\lambda_1(\nu)>-1$ for all $\nu<\nu_1$. 

The minimization problem~\eqref{eq:min_lambda_1} is one of our main motivations for studying Dirac operators with general charge densities $\mu$. Indeed, even if we restrict the minimization to $\mu\in C^\ii_c(\R^3,\R_+)$, a minimizing sequence will always converge to a singular measure, as will be proved in Theorem~\ref{thm:existence_optimal_measure} below.

\subsection{Characterization of the critical couplings}

Our main first result is a characterization of the two numbers in~\eqref{eq:def_nu_c} by a formula based on the Birman-Schwinger principle or on Hardy inequalities. Since we study here $\nu_0$ and $\nu_1$, it is convenient to work with probability measures $\mu$ throughout.

\begin{thm}[The critical coupling constants $\nu_0$ and $\nu_1$]\label{thm:critical_nu's}
We have 
\begin{equation}
 \frac1{\nu_0}=\sup_{\substack{\mu\geq0\\ \mu(\R^3)=1}} \norm{\sqrt{V_\mu}\frac{1}{\alpha\cdot p+\beta}\sqrt{V_\mu}}_{L^2(\R^3,\C^4)\to L^2(\R^3,\C^4)}
 \label{eq:formula_nu_0}
\end{equation}
and
\begin{align}
\frac1{\nu_1}&=\sup_{\substack{\mu\geq0\\ \mu(\R^3)=1}} \norm{\sqrt{V_\mu}\frac{1}{\sigma\cdot p}\sqrt{V_\mu}}_{L^2(\R^3,\C^2)\to L^2(\R^3,\C^2)}\nn\\
&=\sup_{\substack{\mu\geq0\\ \mu(\R^3)=1}} \norm{\sqrt{V_\mu}\frac{1}{\alpha\cdot p}\sqrt{V_\mu}}_{L^2(\R^3,\C^4)\to L^2(\R^3,\C^4)}\nn\\
&=\sup_{\substack{\mu\geq0\\ \mu(\R^3)=1}} \max\;\Spec\left(\sqrt{V_\mu}\frac{1}{\alpha\cdot p+\beta +1}\sqrt{V_\mu}\right).
\label{eq:formula_nu_1}
\end{align}
Here $p=-i\nabla$ and $V_\mu:=\mu\ast|x|^{-1}$. The function $\sqrt{V_\mu}$ is everywhere interpreted as a multiplication operator.
In addition, we have the estimates 
\begin{equation}
0.9\simeq \frac{2}{\frac\pi2+\frac2\pi}\leq\nu_0\leq \nu_1\leq1.
 \label{eq:estim_nu_c}
\end{equation}
\end{thm}

The proof of Theorem~\ref{thm:critical_nu's} is provided later in Section~\ref{sec:proof_critical_nu's}.

\begin{remark}\label{rmk:form_domain}
On the third line of~\eqref{eq:formula_nu_1} we compute the maximum of the spectrum, not the norm. The operator in the parenthesis is not necessarily bounded from below for singular measures $\mu$. Indeed, using that 
$$\frac{1}{\alpha\cdot p+\beta+1}=\frac{\alpha\cdot p+\beta-1}{|p|^2}=\frac1{\alpha\cdot p}+\frac{\beta-1}{|p|^2}$$
and letting again $V_\mu:=\mu\ast|x|^{-1}$, we obtain
\begin{equation}
\sqrt{V_\mu}\frac{1}{\alpha\cdot p+\beta+1}\sqrt{V_\mu}
=\sqrt{V_\mu}\frac{1}{\alpha\cdot p}\sqrt{V_\mu}+(\beta-1)\sqrt{V_\mu}\frac{1}{|p|^2}\sqrt{V_\mu}. 
\label{eq:def_op_D_0+1}
\end{equation}
The first operator is bounded by Kato's inequality
\begin{equation}
 \frac{1}{|x|}\leq\frac\pi2\,|p|
 \label{eq:Kato}
\end{equation}
and the second is non-positive since $\beta\leq1$. The second operator has a unique (Friedrichs) non-positive self-adjoint extension. This is our interpretation of the operator in the parenthesis. 
\end{remark}

\begin{remark}[Estimate on the first eigenvalue for $\mu(\R^3)<\nu_0$]
When $\mu\geq0$ and $\mu(\R^3)<\nu_0$ we can give a simple estimate on the first eigenvalue by using the Birman-Schwinger principle. For $-1\leq E<1$, we write
$$\sqrt{V_\mu}\frac1{D_0-E}\sqrt{V_\mu}=\sqrt{V_\mu}\frac1{D_0}\sqrt{V_\mu}+\sqrt{V_\mu}\frac{E}{(D_0-E)D_0}\sqrt{V_\mu}.$$
Since $|D_0|\geq1$ and $|E|\leq 1$, we have $(D_0-E)D_0\geq(1-|E|)\geq0$ and hence the second term on the right is non-positive for $E\leq0$. Thus, using~\eqref{eq:formula_nu_0} we obtain the estimate 
$$\norm{\sqrt{V_\mu}\frac1{D_0-E}\sqrt{V_\mu}}\leq \frac{\mu(\R^3)}{\nu_0}+\frac{\pi\mu(\R^3)}2 \frac{E_+}{1-E_+}$$
where $E_+=\max(0,E)$. The right side is $<1$ for $\mu(\R^3)<\nu_0$ and 
$$-1\leq E<\frac{\nu_0-\mu(\R^3)}{\frac{\pi}2\nu_0\mu(\R^3)+\nu_0-\mu(\R^3)}.$$
By the Birman-Schwinger principle, this shows that 
\begin{equation}
 \boxed{\lambda_1\left(D_0-\mu\ast\frac1{|x|}\right)\geq \frac{\nu_0-\mu(\R^3)}{\frac{\pi}2\nu_0\mu(\R^3)+\nu_0-\mu(\R^3)}}
 \label{eq:lower_bound_lambda_nu_0}
\end{equation}
for all positive measures $\mu$ such that $\mu(\R^3)<\nu_0$. In Theorem~\ref{thm:signed} below we explain how to use this bound for signed measures. 

Due to the singularity of the resolvent at $-1$ it is not obvious to provide a similar bound for $\mu(\R^3)<\nu_1$.
\end{remark}

The formula~\eqref{eq:formula_nu_1} can be interpreted in the form of Hardy-type inequalities similar to those studied in~\cite{DolEstSer-00,DolEstLosVeg-04,DolEstDuoVeg-07,ArrDuoVeg-13,CasPizVeg-20}. Indeed, denoting  again $V_\mu=\mu\ast|x|^{-1}$ for shortness, the first line in~\eqref{eq:formula_nu_1} means that 
$$\norm{\sqrt{V_\mu}\frac{1}{\sigma\cdot p}\sqrt{V_\mu}u}_{L^2(\R^3,\C^2)}^2\leq \frac{\norm{u}_{L^2(\R^3,\C^2)}^2}{\nu^2_1}\mu(\R^3)^2$$
for all $u\in L^2(\R^3,\C^2)$ and all positive measures $\mu$. Letting $u=V_\mu^{-1/2}\sigma\cdot p\;\phi$ with $\phi\in C^\ii_c(\R^3,\C^2)$, we obtain that $\nu_1$ is also the best constant in the Hardy-type inequality
\begin{equation}
\boxed{  \int_{\R^3}\frac{|\sigma\cdot\nabla\phi|^2}{\mu\ast|x|^{-1}}\,dx\geq\frac{\nu_1^2}{\mu(\R^3)^2}\int_{\R^3}\left(\mu\ast\frac1{|x|}\right)|\phi|^2\,dx}
 \label{eq:nu_c_Hardy_2component}
\end{equation}
for every $\phi\in C^\ii_c(\R^3,\C^2)$ and every positive measure $\mu$ on $\R^3$. From~\eqref{eq:estim_nu_c}, the inequality is known to already hold with the constant $2/(\pi/2+2/\pi)$ instead of $\nu_1^2$. This can also be written in the form
\begin{equation}
\nu_1^2=\inf_{\substack{\mu,\phi\in C^\ii_c(\R^3,\C^2)\\ \phi\neq0,\ \mu\geq0\\ 0<\mu(\R^3)<\ii}}\frac{\dps \mu(\R^3)^2\int_{\R^3}\frac{|\sigma\cdot\nabla\phi|^2}{\mu\ast|x|^{-1}}\,dx}{\dps \int_{\R^3}\left(\mu\ast|x|^{-1}\right)|\phi|^2\,dx}.
\label{eq:nu_c_quotient}
\end{equation}
Similarly, one can see that $\nu_0$ is the best constant in the Hardy-type inequality
\begin{equation}
\boxed{  \int_{\R^3}\frac{|\sigma\cdot\nabla\phi|^2}{1+\mu\ast|x|^{-1}}\,dx\geq\nu_0^2\int_{\R^3}\left(-1+\mu\ast\frac1{|x|}\right)|\phi|^2\,dx}
 \label{eq:nu_0_Hardy_2component}
\end{equation}
for every $\phi\in C^\ii_c(\R^3,\C^2)$ and every \emph{probability} measure $\mu$. 

\medskip

Now we come to our main conjecture which can be expressed as follows. 

\begin{conjecture}[Optimality for a delta]\label{conjecture}
We have 
\begin{equation}
\lambda_1(\nu)=\lambda_1(D_0-\nu|x|^{-1})=\sqrt{1-\nu^2}
\label{eq:conjecture_eigenvalue}
\end{equation}
for all $0\leq\nu<1$. This implies 
\begin{equation}
\nu_0=\nu_1=1.
\label{eq:conjecture_nu}
\end{equation}
\end{conjecture}

The conjecture states that the first eigenvalue $\lambda_1(D_0-V_\mu)$ is minimal when $\mu=\mu(\R^3)\delta_0$, that is, in the pure Coulomb case. Recall that the similar property holds in the Schrödinger case~\eqref{eq:estim_Schrodinger}. A stronger conjecture would be that $\lambda_1(D_0-V_\mu)$ is a concave function of $\mu$, but we do not commit ourselves in this direction since this might fail in the strong relativistic regime, whereas~\eqref{eq:conjecture_eigenvalue} could still remain true. Conjecture~\ref{conjecture} implies immediately an earlier conjecture made in~\cite{EstLewSer-20a_ppt} which was restricted to multi-center Coulomb potentials. As was mentioned in the introduction, the conjecture holds for radially symmetric measures by Newton's theorem and numerical simulations from~\cite{ArtSurIndPluSto-10,McConnell-13} suggest that it also holds for $\mu$ the sum of two identical deltas. 

\subsection{An estimate for signed measures}
Before turning to the properties of the lowest possible eigenvalue $\lambda_1(\nu)$, we mention a useful result concerning the critical number $\nu_0$. The following gives the persistence of a gap for signed measures. 

\begin{thm}[Gap for signed measures]\label{thm:signed}
Let $\mu=\mu_+-\mu_-$ be a signed measure with $\mu_\pm\geq0$ and 
$$\nu_\pm:=\mu_\pm(\R^3)<\nu_0,$$ 
the critical number defined in~\eqref{eq:def_nu_c}. Then the Dirac-Coulomb operator 
$$D_0-\mu\ast\frac1{|x|},$$
as defined in~\cite[Thm.~1]{EstLewSer-20a_ppt}, has the gap around the origin
$$\left(-\lambda_1\left(D_0-\mu_-\ast\frac1{|x|}\right)\;,\; \lambda_1\left(D_0-\mu_+\ast\frac1{|x|}\right)\right)$$
which contains the interval
$$\big(-\lambda_1(\nu_-)\;,\;\lambda_1(\nu_+)\big)$$
where we recall that $\lambda_1(\nu)$ is defined in~\eqref{eq:min_lambda_1}.
\end{thm}

Similar results were proved in~\cite{DolEstSer-06,FouLewTri-20}. Note that the bound~\eqref{eq:lower_bound_lambda_nu_0} implies that the gap also contains the interval
$$\left(-\frac{\nu_0-\nu_-}{\frac{\pi}2\nu_0\nu_-+\nu_0-\nu_-}\;,\; \frac{\nu_0-\nu_+}{\frac{\pi}2\nu_0\nu_++\nu_0-\nu_+}\right).$$
In~\cite{FouLewTri-20} it was even proved that there exists a constant $c=c(\nu_-,\nu_+)>0$ so that $|D_0-\mu\ast|x|^{-1}|\geq c|D_0|$ and this was important for establishing some properties of solutions to Dirac-Fock equations. 

\begin{proof}
One possibility is to use the min-max principle for $D_0-V_\mu$ in the spectral decomposition of $D_0\mp V_{\mu_\pm}$ and a continuation principle in the spirit of~\cite{DolEstSer-00} which states that the eigenvalues of $D_0-(t\mu_+-s\mu_-)\ast|x|^{-1}$ are all decreasing in $t$ and increasing in $s$. We provide here a different proof based on the Birman-Schwinger principle and the resolvent formula. Since $\mu_+(\R^3)<\nu_0$, the operator $D_0-V_{\mu_+}$ has its discrete spectrum included in $(\lambda_1(D_0-V_{\mu_+}),1)$ with $\lambda_1(D_0-V_{\mu_+})>0$. Thus, for $-1<E<\lambda_1(D_0-V_{\mu_+})$
we have 
$$\sqrt{V_{\mu_+}}\frac{1}{D_0-E}\sqrt{V_{\mu_+}}<1,$$
by the Birman-Schwinger principle. Now, following~\cite{Nenciu-76,KlaWus-79,Klaus-80b} we write 
\begin{multline}
\frac{1}{D_0-V_{\mu_+}-E}=\frac{1}{D_0-E}\\+\frac{1}{D_0-E}\sqrt{V_{\mu_+}}\frac{1}{1-\dps\sqrt{V_{\mu_+}}\frac{1}{D_0-E}\sqrt{V_{\mu_+}}}\sqrt{V_{\mu_+}}\frac{1}{D_0-E}
\label{eq:resolvent_V_+}
\end{multline}
and notice that the second operator on the right of~\eqref{eq:resolvent_V_+} is non-negative. Thus we have shown the operator inequality
$$ \frac{1}{D_0-V_{\mu_+}-E}\geq\frac{1}{D_0-E}$$
(this inequality does not immediately follows from the fact that $x\mapsto x^{-1}$ is operator-monotone since the operators have no sign). Multiplying by $\sqrt{V_{\mu_-}}$ on both sides, we obtain 
\begin{equation}
 \sqrt{V_{\mu_-}}\frac{1}{D_0-V_{\mu_+}-E}\sqrt{V_{\mu_-}}\geq \sqrt{V_{\mu_-}}\frac{1}{D_0-E}\sqrt{V_{\mu_-}}.
 \label{eq:compare_resolvents}
\end{equation}
By charge-conjugation, the operator  $D_0+V_{\mu_-}$ has its discrete spectrum included in $(-1,-\lambda_1(D_0-V_{\mu_-}))$ with $\lambda_1(D_0-V_{\mu_-})>0$ and the Birman-Schwinger principle now tells us that 
$$\sqrt{V_{\mu_-}}\frac{1}{D_0-E}\sqrt{V_{\mu_-}}>-1$$
for $-\lambda_1\left(D_0+V_{\mu_-}\right)<E<1$. Hence, after inserting in~\eqref{eq:compare_resolvents}, we find
$$\sqrt{V_{\mu_-}}\frac{1}{D_0-V_{\mu_+}-E}\sqrt{V_{\mu_-}}>-1$$
for all 
$$-\lambda_1\left(D_0-V_{\mu_-}\right)<E<\lambda_1\left(D_0-V_{\mu+}\right).$$
By the Birman-Schwinger principle using the operator $D_0-V_{\mu_+}$ as a reference, this proves that $D_0-V_\mu$ has no eigenvalue in the interval mentioned in the statement.
\end{proof}

\subsection{Existence of an optimal measure}

Our main result in this article concerns the existence of an optimal measure for the variational problem $\lambda_1(\nu)$ defined in~\eqref{eq:min_lambda_1} and all sub-critical coupling constant $0\leq \nu<\nu_1$, with $\nu_1$ as in~\eqref{eq:def_nu_c}.

\begin{thm}[Optimal measure]\label{thm:existence_optimal_measure}
We have the following results:

\medskip

\noindent $(i)$ The function $\nu\mapsto \lambda_1(\nu)$ is Lipschitz-continuous on $[0,\nu_1)$, decreasing and takes its values in $(-1,1]$ with $\lambda_1(0)=1$.

\medskip

\noindent $(ii)$ For any $\nu\in[0,\nu_1)$, there exists a positive measure $\mu_\nu$ with $\mu_\nu(\R^3)=\nu$ so that 
$$\lambda_1\left(D_0-\mu_\nu\ast\frac1{|x|}\right)=\lambda_1(\nu).$$
More precisely, any minimizing sequence $\{\mu_n\}$ for $\lambda_1(\nu)$ is tight up to space translations and converges tightly to an optimal measure for $\lambda_1(\nu)$. 

\medskip

\noindent $(iii)$ Any such minimizer $\mu_\nu$ concentrates on the compact set
\begin{equation}
K:=\left\{x\in\R^3\ :\ |\Psi_\nu|^2\ast\frac1{|\cdot|}(x)=\max_{\R^3} \left(|\Psi_\nu|^2\ast\frac1{|\cdot|}\right)\right\}
\label{eq:Euler-Lagrange}
\end{equation}
where $\Psi_\nu$ is any eigenfunction of $D_0-V_{\mu_\nu}$ of eigenvalue $\lambda_1(\nu)$. The compact set $K$ in~\eqref{eq:Euler-Lagrange} has a zero Lebesgue measure, hence $\mu_\nu$ is singular with respect to the Lebesgue measure.
\end{thm}

The theorem is proved later in Section~\ref{sec:proof_existence}, where the reader will also find a description of the main ideas of the proof. Note that the potential $|\Psi_\nu|^2\ast|x|^{-1}$ is a continuous function tending to zero at infinity, since we know that $\Psi_\nu\in H^{1/2}(\R^3,\C^4)$ by~\cite[Thm.~1]{EstLewSer-20a_ppt}. Hence the set $K$ in~\eqref{eq:Euler-Lagrange} is compact. If we knew that this function attains its maximum at a unique point (that is, $K$ is reduced to one point), we would deduce that $\mu_\nu$ is proportional to a delta measure, as we conjecture. We can prove the weaker statement that $K$ has zero Lebesgue measure, using a unique continuation principle proved in Appendix~\ref{app:unique_continuation}. That the optimal measure $\mu_\nu$  is necessarily singular was our main motivation for studying Dirac operators with general charge distributions. 

The rest of the article is devoted to the proof of our main results.

\section{Characterization of $\nu_0$ and $\nu_1$}\label{sec:proof_critical_nu's}

Our goal in this section is to prove Theorem~\ref{thm:critical_nu's} about the critical numbers $\nu_0$ and $\nu_1$ defined in~\eqref{eq:def_nu_c}. The proof essentially follows from the Birman-Schwinger principle recalled in Section~\ref{sec:rappels} and based on the operator $K_\lambda$ in~\eqref{eq:def_K_V}. However, the argument is more delicate for $\nu_1$ since we are in the situation of an eigenvalue approaching the essential spectrum. In Section~\ref{sec:prop_nu_1} we start by establishing some spectral properties of $K_0$ which will be useful for the proof of Theorem~\ref{thm:critical_nu's}, provided in Section~\ref{sec:proof_thm_nu's}.

We will use that for Coulomb potentials 
\begin{multline}
\Spec\left(\frac{1}{|x|^{\frac12}}\frac{1}{\alpha\cdot p+\beta}\frac{1}{|x|^{\frac12}}\right)=\Spec_{\rm ess}\left(\frac{1}{|x|^{\frac12}}\frac{1}{\alpha\cdot p+\beta}\frac{1}{|x|^{\frac12}}\right)\\=\Spec\left(\frac{1}{|x|^{\frac12}}\frac{1}{\alpha\cdot p}\frac{1}{|x|^{\frac12}}\right)=\Spec_{\rm ess}\left(\frac{1}{|x|^{\frac12}}\frac{1}{\alpha\cdot p}\frac{1}{|x|^{\frac12}}\right)=[-1,1] 
\label{eq:rappel_Coulomb}
\end{multline}
and that 
\begin{equation}
\norm{\frac{1}{|x|^{\frac12}}\frac{1}{\alpha\cdot p+is}\frac{1}{|x|^{\frac12}}}=\norm{\frac{1}{|x|^{\frac12}}\frac{1}{\alpha\cdot p+\beta+is}\frac{1}{|x|^{\frac12}}}=1.
\label{eq:rappel_Coulomb_2}
\end{equation}
for all $s\in\R$. See~\cite{Nenciu-76,Wust-77,Klaus-80b,Kato-83,ArrDuoVeg-13}. We also recall Tix's inequality~\cite{Tix-98}
\begin{equation}
\norm{\frac1{|x|^{\frac12}}\frac{P_0^\pm}{\sqrt{1-\Delta}}\frac1{|x|^{\frac12}}}=\norm{\frac{P_0^\pm}{\sqrt{1-\Delta}}\frac1{|x|}\frac{P_0^\pm}{\sqrt{1-\Delta}}}\leq \frac{\frac\pi2+\frac2\pi}{2}
\label{eq:Tix}
\end{equation}
where $P_0^\pm=\1_{\R_\pm}(D_0)$ are the free Dirac spectral projections. Throughout this section, $\mu$ is by convention always taken to be a probability measure. 

\subsection{Some spectral properties of $K_\lambda$}\label{sec:prop_nu_1}

We need two preliminary lemmas, whose proofs are in fact the only places where we use some technical arguments from~\cite{EstLewSer-20a_ppt}. The first is about the essential spectrum of $K_\lambda$. 

\begin{lemma}[Essential spectrum of $K_\lambda$]\label{lem:essential_spectrum}
Let $\mu$ be any probability measure and $\nu_{\rm max}(\mu):=\max_{R\in\R^3}\mu(\{R\})\leq1$. Then we have
\begin{equation}
\Spec_{\rm ess}\left(\sqrt{\mu\ast\frac1{|x|}}\frac1{\alpha\cdot p+\eps\beta-\lambda}\sqrt{\mu\ast\frac1{|x|}}\right)=\big[-\nu_{\rm max}(\mu),\nu_{\rm max}(\mu)\big]
\label{eq:essential_spectra}
\end{equation}
for all $\eps>0$ and $|\lambda|<\eps$, as well as 
\begin{equation}
\Spec_{\rm ess}\left(\sqrt{\mu\ast\frac1{|x|}}\frac1{\alpha\cdot p}\sqrt{\mu\ast\frac1{|x|}}\right)=[-1,1].
\label{eq:essential_spectra_eps_0}
\end{equation}
\end{lemma}

\begin{proof}[Proof of Lemma~\ref{lem:essential_spectrum}]
Noticing that 
$$\frac1{\alpha\cdot p+\eps\beta-\lambda}=\frac1{\alpha\cdot p+\eps\beta}+\frac\lambda{(\alpha\cdot p+\eps\beta-\lambda)(\alpha\cdot p+\eps\beta)}$$
with
$$\left|\frac\lambda{(\alpha\cdot p+\eps\beta-\lambda)(\alpha\cdot p+\eps\beta)}\right|\leq \frac{C_{\lambda,\eps}}{1+|p|^2}$$
for some constant $C_{\lambda,\eps}$, we see that 
$$\sqrt{V_\mu}\frac1{\alpha\cdot p+\eps\beta-\lambda}\sqrt{V_\mu}-\sqrt{V_\mu}\frac1{\alpha\cdot p+\eps\beta}\sqrt{V_\mu}$$
is compact. Thus the essential spectrum is invariant and it suffices to prove~\eqref{eq:essential_spectra} for $\lambda=0$, which we assume throughout the rest of the proof. 

We write
$$\mu=\sum_{m=1}^M\nu_m\delta_{R_m}+\tilde\mu$$
where $M$ can be infinite and $\tilde\mu$ has no atom. Truncating both the sum and $\tilde\mu$ in space and using Kato's inequality~\eqref{eq:Kato}, we see that it suffices to prove the lemma for a finite sum and for $\tilde\mu$ of compact support, all included in the ball of radius $N$.
For simplicity of notation we still assume that $\mu(\R^3)=1$. We have the pointwise estimate
\begin{equation}
\frac{1}{|x|+N}\leq V_\mu(x)\leq \frac{1}{|x|-N}
\label{eq:behavior_V_infinity}
\end{equation}
for $|x|>N$ which proves that $|x|V_\mu(x)\to 1$ at infinity. 

Note that if $\mu$ has no atom ($\nu_{\rm max}(\mu)=0$) then we know that the operators are compact by~\cite[Lem.~8]{EstLewSer-20a_ppt}, which proves~\eqref{eq:essential_spectra} in this case. Also, the result is well known when $\mu$ is a delta measure, see~\eqref{eq:rappel_Coulomb}. Concentrating trial functions at one of the deltas we see that 
\begin{equation}
\big[-\nu_{\rm max}(\mu),\nu_{\rm max}(\mu)\big]\subset \Spec_{\rm ess}\left(\sqrt{V_\mu}\frac1{\alpha\cdot p+\eps\beta}\sqrt{V_\mu}\right)
\label{eq:inclusion_sp_ess}
\end{equation}
and our main task will be to derive the other inclusion. 

In the case $\eps=0$ we can also dilate functions. By~\eqref{eq:rappel_Coulomb} and a density argument, there exists a (normalized) Weyl sequence $\Psi_n\in C^\ii_c(\R^3\setminus\{0\},\C^4)$ so that $\pscal{\Psi_n,|x|^{-1/2}( \alpha\cdot p)^{-1}|x|^{-1/2}\Psi_n}\to\lambda\in[-1,1]$. By dilating $\Psi_n$ and using the scaling invariance of $|x|^{-1/2}( \alpha\cdot p)^{-1}|x|^{-1/2}$, we can assume that $\Psi_n$ is supported outside of a ball $B_{r_n}$ with $r_n\to\ii$. Next, we write
$$\pscal{\Psi_n,\sqrt{V_\mu}\frac1{\alpha\cdot p}\sqrt{V_\mu}\Psi_n} =\pscal{|x|^{\frac12}\sqrt{V_\mu}\Psi_n,\left(\frac1{|x|^{\frac12}}\frac1{\alpha\cdot p}\frac1{|x|^{\frac12}}\right)|x|^{\frac12}\sqrt{V_\mu}\Psi_n}$$
and we use that 
$$\norm{\big(|x|^{\frac12}\sqrt{V_\mu}-1\big)\Psi_n}_{L^2(\R^3,\C^4)}\leq \frac{C}{r_n}\underset{n\to\ii}\longrightarrow0$$ 
because $||x|^{1/2}\sqrt{V_\mu}-1|\leq C/r_n$ on the support of $\Psi_n$, by~\eqref{eq:behavior_V_infinity}. Since the operator $|x|^{-1/2}( \alpha\cdot p)^{-1}|x|^{-1/2}$ is bounded, we deduce that 
$$\lim_{n\to\ii}\pscal{\Psi_n,\sqrt{V_\mu}\frac1{\alpha\cdot p}\sqrt{V_\mu}\Psi_n}=\lim_{n\to\ii}\pscal{\Psi_n,\frac1{|x|^{\frac12}}\frac1{\alpha\cdot p}\frac1{|x|^{\frac12}}\Psi_n}=\lambda.$$
Thus we have constructed a Weyl sequence for the operator $\sqrt{V_\mu}( \alpha\cdot p)^{-1}\sqrt{V_\mu}$ and we conclude after varying $\lambda$ that 
\begin{equation}
 [-1,1]\subset \Spec_{\rm ess}\left(\sqrt{V_\mu}\frac1{\alpha\cdot p}\sqrt{V_\mu}\right).
 \label{eq:inclusion_sp_ess2}
\end{equation}

Now we discuss the converse inclusions. Similarly as in the proof of~\cite[Thm.~1]{EstLewSer-20a_ppt}, we consider the following partition of unity:
$$1=\sum_{m=1}^M\1_{B_\eta(R_m)}+\1_{B_R\setminus\cup_{m=1}^MB_\eta(R_m)}+\1_{\R^3\setminus B_R}$$
where $R$ is chosen large enough and $\eta$ is chosen small enough, so that the balls $B_\eta(R_m)$ do not intersect and are all included in $B_{R/2}$. We insert our partition of unity on both sides of our operator and expand. We claim that all the cross terms are compact, so that
\begin{align}
 A_\eps&:=\sqrt{V_\mu}\frac1{\alpha\cdot p+\eps\beta}\sqrt{V_\mu}\nn\\
 &= \sum_{m=1}^M\1_{B_\eta(R_m)}A_\eps\1_{B_\eta(R_m)}+\1_{\R^3\setminus B_R}A_\eps\1_{\R^3\setminus B_R}+\cK
 \label{eq:decompose_operator_essential_spectrum}
\end{align}
where $\cK$ is compact. For instance, the compactness of 
$$\1_{B_\eta(R_m)}A_\eps\1_{B_\eta(R_\ell)}$$
with $\ell\neq m$ follows from the same proof as in~\cite[Lem.~7]{EstLewSer-20a_ppt}. The functions $\1_{B_\eta(R_\ell)}\sqrt{V_\mu}$ are in $L^2$ and the operator $(\alpha \cdot p+\eps \beta)^{-1}$ has the kernel
$$(\alpha \cdot p+\eps \beta)^{-1}(x,y)=\left(-i\alpha\cdot\nabla_x+\eps\beta\right)\frac{e^{-\sqrt{\eps}|x-y|}}{4\pi\,|x-y|}.$$
This is exponentially decaying at infinity for $\eps>0$ and equal to
$$(\alpha \cdot p)^{-1}(x,y)=i\frac{\alpha\cdot(x-y)}{4\pi\,|x-y|^3}$$
when $\eps=0$. Similarly, 
$$\1_{\R^3\setminus B_R}A_\eps\1_{B_\eta(R_m)}$$
is compact because $V$ behaves like $1/|x|$ at infinity and
$$\int_{\R^3\setminus B_R}\frac{\left|(\alpha \cdot p+\eps\beta)^{-1}(x,0)\right|^2}{|x|}\,dx<\ii.$$
When $\eps>0$ the integrand is exponentially decaying whereas when $\eps=0$ it behaves like $|x|^{-5}$. Finally, the terms involving $\1_{B_R\setminus\cup_{m=1}^MB_\eta(R_m)}$ are easier to treat since in this intermediate region the potential induced by the pointwise charges is equal to that of a regularized measure, by Newton's theorem:
$$V_\mu=\left(\frac1{|B_{\eta/2}|}\sum_{m=1}^M\nu_m\1_{B_{\eta/2}(R_m)}+\tilde\mu\right)\ast\frac1{|x|}\qquad\text{on $\R^3\setminus\cup_{m=1}^MB_\eta(R_m)$.}$$
Then 
$$\frac{1}{|p|^{\frac12}}\sqrt{V_\mu}\1_{B_R\setminus\cup_{m=1}^MB_\eta(R_m)}$$
is compact by~\cite[Lem.~8]{EstLewSer-20a_ppt}. This is also why the diagonal term does not appear in~\eqref{eq:decompose_operator_essential_spectrum}. By the same argument we can actually infer that 
$$\1_{B_\eta(R_m)}A_\eps\1_{B_\eta(R_m)}=\nu_m\frac{\1_{B_\eta(R_m)}}{|x-R_m|^{\frac12}}\frac{1}{\alpha\cdot p+\eps\beta}\frac{\1_{B_\eta(R_m)}}{|x-R_m|^{\frac12}}+\cK$$
where $\cK$ is compact. Therefore, we have shown that 
\begin{align}
A_\eps&=\sqrt{V_\mu}\frac1{\alpha\cdot p+\eps\beta}\sqrt{V_\mu}\nn\\
 &= \sum_{m=1}^M\nu_m\frac{\1_{B_\eta(R_m)}}{|x-R_m|^{\frac12}}\frac{1}{\alpha\cdot p+\eps\beta}\frac{\1_{B_\eta(R_m)}}{|x-R_m|^{\frac12}}\nn\\
 &\qquad+\1_{\R^3\setminus B_R}\sqrt{V_\mu}\frac{1}{\alpha\cdot p+\eps\beta}\sqrt{V_\mu}\1_{\R^3\setminus B_R}+\cK_{R,\eta,\eps} \label{eq:decomp_A_compact}
 \end{align}
where $\cK_{R,\eta,\eps}$ is compact.  By~\eqref{eq:rappel_Coulomb_2} we have the operator bound
 $$\frac{\1_{B_\eta(R_m)}}{|x-R_m|^{\frac12}}\frac{1}{\alpha\cdot p+\eps\beta}\frac{\1_{B_\eta(R_m)}}{|x-R_m|^{\frac12}}\leq \1_{B_\eta(R_m)}$$
 and therefore we infer 
\begin{multline}
A_\eps\leq \nu_{\rm max}(\mu)\1_{\cup_{m=1}^MB_\eta(R_m)}\\
+\1_{\R^3\setminus B_R}\sqrt{V_\mu}\frac{1}{\alpha\cdot p+\eps\beta}\sqrt{V_\mu}\1_{\R^3\setminus B_R}+\cK_{R,\eta,\eps}.
\label{eq:estim_A_eps}
\end{multline}
When $\eps>0$ the operator 
 $$\1_{\R^3\setminus B_R}\sqrt{V_\mu}\frac{1}{\alpha\cdot p+\eps\beta}\sqrt{V_\mu}\1_{\R^3\setminus B_R}$$
is also compact. Let $\Psi_n\wto0$ be a Weyl sequence such that $(A_\eps-\lambda)\Psi_n\to0$ with $\lambda:=\max\Spec_{\rm ess}(A_\eps)$. Taking the expectation of~\eqref{eq:estim_A_eps} with $\Psi_n$ and using that the compact terms tend to 0, we obtain $\lambda\leq \nu_{\rm max}(\mu)$. By charge-conjugation invariance, the spectrum and essential spectrum of $A_\eps$ are symmetric with respect to the origin and hence, together with~\eqref{eq:inclusion_sp_ess}, this proves~\eqref{eq:essential_spectra}. 

When $\eps=0$ we instead use the behavior at infinity of $V_\mu$ and~\eqref{eq:rappel_Coulomb_2} to infer that 
 $$\1_{\R^3\setminus B_R}\sqrt{V_\mu}\frac{1}{\alpha\cdot p}\sqrt{V_\mu}\1_{\R^3\setminus B_R}\leq \frac{\1_{\R^3\setminus B_R}}{1-\frac{N}{R}}$$
 where ${\rm supp}(\mu)\subset B_N$ and obtain by a similar reasoning
$$A_\eps\leq \frac{1}{1-\frac{N}{R}}+\cK_{R,\eta,\eps}$$
since $\nu_m\leq1$ for all $m$. After taking $R\to\ii$, this proves the inclusion opposite to~\eqref{eq:inclusion_sp_ess2} and concludes the proof of~\eqref{eq:essential_spectra_eps_0}.
\end{proof}

In~\eqref{eq:decomp_A_compact} we have introduced the compact operator $\cK_{R,\eta,\eps}$. The following provides its limit as $\eps\to0$.

\begin{lemma}[Behavior of $\cK_{R,\eta,\eps}$]\label{lem:CV_K_R_eta_eps}
The operator $\cK_{R,\eta,\eps}$ in~\eqref{eq:decomp_A_compact} converges \emph{in norm} to the corresponding compact operator $\cK_{R,\eta,0}$ when $\eps\to0^+$. 
\end{lemma}

\begin{proof}
The operator $\cK_{R,\eta,\eps}$ can be written in the form 
$$\cK_{R,\eta,\eps}=\sum\sqrt{V_j}(\alpha\cdot p+\eps\beta)^{-1}\sqrt{V_j'}$$
where for each $j$, we have that either $V_j$ or $V_j'$ has compact support, hence belongs to $L^r$ for all $1\leq r<3$. In addition, the supports of $V_j$ and $V_j'$ do not intersect, except for only one term involving $W=\1_{B_R\setminus\cup_{m=1}^MB_\eta(R_m)}\sqrt{V_\mu}$ twice. The terms involving $W$ are rather easy to deal with, since they can be written in the form
$$\sqrt{V_j}\frac1{\alpha\cdot p+\eps\beta}\sqrt{W}=\sqrt{V_j}\frac1{|p|^{\frac12}}\frac{|p|}{\alpha\cdot p+\eps\beta}\frac1{|p|^{\frac12}}\sqrt{W}$$
and since $|p|^{-1/2}\sqrt{W}$ is compact by~\cite[Lem.~8]{EstLewSer-20a_ppt}, the convergence holds in norm. Therefore, we only have to treat the case where $V_j$ and $V_j'$ correspond to either two disjoint balls around some nuclei, or one such ball and the potential $V\1_{\R^3\setminus B_R}$. 

In order to deal with these more complicated terms, it is convenient to use pointwise kernel bounds like in~\cite[Lem.~7]{EstLewSer-20a_ppt}. Note that operator bounds are not very useful since we may have $V_j\neq V_j'$. Recall also that if $|A(x,y)|\leq B(x,y)$, then $\|A\|\leq\|B\|$. 
First we compute the kernel of the difference
\begin{multline}
\left(\frac{1}{\alpha\cdot p+\eps\beta}-\frac1{\alpha\cdot p}\right)(x,y)=-i\frac{\alpha\cdot(x-y)}{4\pi|x-y|^3}\left(1-e^{-\sqrt{\eps}|x-y|}\right)\\
-i\sqrt{\eps}\frac{\alpha\cdot(x-y)}{4\pi|x-y|^2}e^{-\sqrt{\eps}|x-y|}+\eps\beta \frac{e^{-\sqrt{\eps}|x-y|}}{4\pi|x-y|}.
\label{eq:comput_kernel_difference}
\end{multline}
Using for instance that 
$$\frac{1-e^{-r}}{r^2}\leq \frac1{r^{\frac32}},\qquad \frac{e^{-r}}{r}\leq\frac1{r^{\frac32}}$$
we obtain the crude but simple bound for $\eps$ small enough
\begin{equation*}
\left|\left(\frac{1}{\alpha\cdot p+\eps\beta}-\frac1{\alpha\cdot p}\right)(x,y)\right|\leq \frac{3\eps^{\frac14}}{4\pi|x-y|^{\frac32}}.
\end{equation*}
In the case of two non-overlapping balls around two different singularities, $|x-y|$ stays bounded and never vanishes. Hence we find by~\cite[Lem.~7]{EstLewSer-20a_ppt}
$$\norm{\1_{B_\eta(R_m)}\sqrt{V_\mu}\left(\frac{1}{\alpha\cdot p+\eps\beta}-\frac{1}{\alpha\cdot p}\right)\1_{B_\eta(R_\ell)}\sqrt{V_\mu}}\leq C\eps^{\frac14}$$
  with $m\neq\ell$ (the bound can be improved to $\sqrt\eps$).  For the cross term involving one singularity and $V\1_{\R^3\setminus B_R}$ we obtain
\begin{multline}
\norm{\1_{\R^3\setminus B_R}\sqrt{V_\mu}\left(\frac{1}{\alpha\cdot p+\eps\beta}-\frac{1}{\alpha\cdot p}\right)\1_{B_\eta(R_\ell)}\sqrt{V_\mu}}\\
\leq C\eps^{\frac14} \left(\int_{\R^3\setminus B_R}\frac{dx}{|x|^4}\right)^{\frac12}.
\end{multline}
which concludes the proof that $\cK_{R,\eta,\eps}\to \cK_{R,\eta,0}$ in norm. 
\end{proof}

After these preparations we are able to prove the following result, which will be the main ingredient for the proof of Theorem~\ref{thm:critical_nu's} in the case of the critical number $\nu_1$.

\begin{prop}\label{prop:nu_1}
For every probability measure $\mu$, we have
\begin{equation}
\lim_{\eps\to0^+}\max\;\Spec\left(\sqrt{V_\mu}\frac1{D_0+1-\eps}\sqrt{V_\mu}\right)\leq \norm{\sqrt{V_\mu}\frac1{\alpha\cdot p}\sqrt{V_\mu}}.
\label{eq:limit_max_Spec}
\end{equation}
\end{prop}

Using the Birman-Schwinger principle we will see later that the limit on the left of~\eqref{eq:limit_max_Spec} is just $1/\nu_1(\mu)$. This limit exists because the function is increasing.

\begin{proof}[Proof of Proposition~\ref{prop:nu_1}]
For $0<\eps\leq1$, we write
\begin{align}
 \frac{1}{D_0+1-\eps}&=\frac{\alpha\cdot p+\beta-1+\eps}{|p|^2+\eps(2-\eps)}+\nn\\
 &=\frac{1}{\alpha\cdot p+\beta\sqrt{\eps(2-\eps)}}+\frac{(1-\sqrt{\eps(2-\eps)})\beta-1+\eps}{|p|^2+\eps(2-\eps)}\nn\\
 &\leq \frac{1}{\alpha\cdot p+\beta\sqrt{\eps(2-\eps)}}
 \label{eq:decomp_D_0+1}
\end{align}
where we have used here that 
$$(1-\sqrt{\eps(2-\eps)})\beta-1+\eps\leq \eps -\sqrt{\eps(2-\eps)}\leq0.$$
We obtain the operator inequality
\begin{equation}
\sqrt{V_\mu}\frac{1}{D_0+1-\eps}\sqrt{V_\mu}\leq \sqrt{V_\mu}\frac{1}{\alpha\cdot p+\sqrt{\eps(2-\eps)}\beta}\sqrt{V_\mu}
\label{eq:estim_max_spectrum}
\end{equation}
and it implies 
$$\max\Spec\left(\sqrt{V_\mu}\frac{1}{D_0+1-\eps}\sqrt{V_\mu}\right)\leq \norm{\sqrt{V_\mu}\frac{1}{\alpha\cdot p+\sqrt{\eps(2-\eps)}\beta}\sqrt{V_\mu}}.$$
Now we show that 
\begin{equation}
 \lim_{\eps\to0^+}\norm{\sqrt{V_\mu}\frac{1}{\alpha\cdot p+\eps\beta}\sqrt{V_\mu}}=\norm{\sqrt{V_\mu}\frac{1}{\alpha\cdot p}\sqrt{V_\mu}}.
 \label{eq:CV_norm}
\end{equation}
Note that we cannot expect that the operator on the left of~\eqref{eq:CV_norm} converges in norm to the operator on the right. For instance, in the Coulomb case $\mu=\delta_0$, we have by scaling 
$$\norm{\frac1{|x|^{\frac12}}\left(\frac{1}{\alpha\cdot p+\eps\beta}-\frac{1}{\alpha\cdot p}\right)\frac1{|x|^{\frac12}}}=
\norm{\frac1{|x|^{\frac12}}\left(\frac{1}{\alpha\cdot p+\beta}-\frac{1}{\alpha\cdot p}\right)\frac1{|x|^{\frac12}}}$$
for all $\eps>0$ and this certainly does not converge to 0. However, in the Coulomb case the two norms in~\eqref{eq:CV_norm} are equal to 1, as recalled in~\eqref{eq:rappel_Coulomb_2}.

We will use that 
$$A_\eps:=\sqrt{V_\mu}\frac{1}{\alpha\cdot p+\eps\beta}\sqrt{V_\mu}\underset{\eps\to0^+}{\longrightarrow}\sqrt{V_\mu}\frac{1}{\alpha\cdot p}\sqrt{V_\mu}=:A_0$$
strongly in the operator sense, that is, $A_\eps\Psi\to A_0\Psi$ for any $\Psi\in L^2(\R^3,\C^4)$. In fact we have $\norm{A_\eps}\leq \pi/2$ for all $\eps\geq0$ by Kato's inequality~\eqref{eq:Kato} and the limit holds when $\Psi\in C^\ii_c(\R^3,\C^4)$, hence must hold everywhere by density. The maximum of the spectrum of a bounded self-adjoint operator 
\begin{equation}
\max\Spec(A)=\sup_{\|\Psi\|=1}\pscal{\Psi,A\Psi}
\label{eq:variational_max_spec}
\end{equation}
is lower semi-continuous for the strong operator convergence. The spectrum of $A_\eps$ is symmetric with respect to the origin by charge conjugation symmetry and hence in our case $\max\Spec(A_\eps)=\|A_\eps\|$. Thus we conclude that 
\begin{equation}
 \norm{\sqrt{V_\mu}\frac{1}{\alpha\cdot p}\sqrt{V_\mu}}\leq \liminf_{\eps\to0^+}\norm{\sqrt{V_\mu}\frac{1}{\alpha\cdot p+\eps\beta}\sqrt{V_\mu}}.
 \label{eq:argument_norm_wlsc}
\end{equation}

It remains to prove the converse inequality. We argue by contradiction and assume that, after extracting a limit, $\|A_{\eps_n}\|\to \lambda>\|A_0\|$.  Since $\|A_0\|\geq1$ by Lemma~\ref{lem:essential_spectrum}, this implies in particular that $\|A_{\eps_n}\|=:\lambda_n$ is an eigenvalue of $A_{\eps_n}$ for $n$ large enough. Let $u_n$ be a corresponding normalized eigenvector:
$$A_{\eps_n} u_n=\lambda_n \, u_n.$$
After extracting another subsequence, we may assume that $u_n\wto u$ weakly. Passing to the weak-$\ast$ limit (using the strong convergence of $A_{\eps_n}$), we obtain 
$A_0u=\lambda u$ and therefore $u=0$ since $\lambda>\|A\|$. Now we go back to~\eqref{eq:decomp_A_compact} and use that
$$\lim_{n\to\ii}\pscal{u_n,\cK_{R,\eta,\eps_n}u_n}=\lim_{n\to\ii}\pscal{u_n,\cK_{R,\eta,0}u_n}=0$$
due to the norm convergence $\cK_{R,\eta,R,\eps}\to \cK_{R,\eta,0}$ from Lemma~\ref{lem:CV_K_R_eta_eps} and the compactness of $\cK_{R,\eta,0}$. We find
$$\lambda=\lim_{n\to\ii}\pscal{u_n,A_{\eps_n}u_n}\leq \frac{1}{1-\frac{N}{R}}.$$
Taking $R\to\ii$ we conclude that $\lambda\leq1$, a contradiction. Therefore we have proved~\eqref{eq:CV_norm} and this concludes the proof of Proposition~\ref{prop:nu_1}.
\end{proof}

After these preliminary results we turn to the proof of Theorem~\ref{thm:critical_nu's}.

\subsection{Proof of Theorem~\ref{thm:critical_nu's}}\label{sec:proof_thm_nu's}
The proof follows essentially from the Birman-Schwinger principle, with some technical difficulties for $\nu_1$. 

\subsubsection*{Step 1. Proof of~\eqref{eq:estim_nu_c}}

The lower bound in~\eqref{eq:estim_nu_c} follows immediately from the definition of $\nu_1$ and the statement in~\cite[Thm.~3]{EstLewSer-20a_ppt} that $\lambda_1(D_0-V_\mu)>0$ for $\mu(\R^3)<2/(\pi/2+2/\pi)$, due to Tix's inequality~\eqref{eq:Tix}. Also, by definition we have  $\nu_0(\mu)\leq\nu_1(\mu)$ for every $\mu$, and therefore $\nu_0\leq\nu_1$. Finally, we have already seen that $\nu_1\leq\nu_1(\delta_0)=1$. 

\subsubsection*{Step 2. Proof of~\eqref{eq:formula_nu_0}}

Recall the Birman-Schwinger principle from~\cite[Thm.~3]{EstLewSer-20a_ppt} which tells us that $\lambda$ is an eigenvalue of $D_0-t V_\mu$ if and only if $1/t$ is an eigenvalue of the bounded operator $K_\lambda=\sqrt{V_\mu}(D_0-\lambda)^{-1}\sqrt{V_\mu}$. The ordered eigenvalues of this operator (outside of the essential spectrum) are increasing with $\lambda$ and Lipschitz (since $z\mapsto K_z$ is an analytic family of bounded operators, they are indeed real analytic curves that may cross). We conclude that $\lambda$ is the first eigenvalue of $D_0-tV_\mu$ if and only if $1/t$ is in fact the largest eigenvalue of $K_\lambda$. 

Due to our definition~\eqref{eq:def_nu_c} of $\nu_0(\mu)$ there are two situations to be considered for the decreasing function $t\mapsto \lambda_1(D_0-tV_\mu)$. If the eigenvalue crosses $0$ before $t$ reaches $1/\nu_{\rm max}(\mu)$, then the Birman-Schwinger principle provides 
\begin{equation}
\frac1{\nu_0(\mu)}=\max\;\Spec\left(\sqrt{V_\mu}\frac1{D_0}\sqrt{V_\mu}\right)=\norm{\sqrt{V_\mu}\frac1{D_0}\sqrt{V_\mu}}.
\label{eq:formula_BS_nu_0}
\end{equation}
The last equality holds because the spectrum is symmetric, by charge-conjugation.
The second situation is when $t$ reaches $1/\nu_{\rm max}(\mu)$ before the eigenvalue crosses the origin. From the Birman-Schwinger principle this means that necessarily
$$\max\;\Spec\left(\sqrt{V_\mu}\frac1{D_0}\sqrt{V_\mu}\right)\leq \nu_{\rm max}(\mu)$$
otherwise  the eigenvalue would have crossed the origin earlier. But by Lemma~\ref{lem:essential_spectrum} we know that the essential spectrum of the operator on the left is $[-\nu_{\rm max}(\mu),\nu_{\rm max}(\mu)]$ so that the maximum of the spectrum is always larger than or equal to $\nu_{\rm max}(\mu)$. Thus there must be equality and we conclude that~\eqref{eq:formula_BS_nu_0} holds in all cases. This proves~\eqref{eq:formula_nu_0} since $\nu_0=\inf_{\mu}\nu_0(\mu)$. In fact we have also shown that 
$$\lim_{t\to\nu_0(\mu)^-}\lambda_1(D_0-tV_\mu)=0.$$

\subsubsection*{Step 3. Proof of~\eqref{eq:formula_nu_1}}

The argument for $\nu_1(\mu)$ is similar but a little more subtle since we are approaching the lower essential spectrum. We have by Lemma~\ref{lem:essential_spectrum}
$$\Spec_{\rm ess}\left(\sqrt{V_\mu}\frac{1}{D_0+1-\eps}\sqrt{V_\mu}\right)=\big[-\nu_{\rm max}(\mu),\nu_{\rm max}(\mu)\big]$$
and therefore
\begin{equation}
 \max\;\Spec\left(\sqrt{V_\mu}\frac1{D_0+1-\eps}\sqrt{V_\mu}\right)\geq \nu_{\rm max}(\mu)
 \label{eq:lower_bound_max_Spec}
\end{equation}
for all $0<\eps<2$. If $\lambda_1(D_0-tV_\mu)$ approaches $-1$ before $t$ reaches $1/\nu_{\rm max}(\mu)$, then the Birman-Schwinger principle provides 
\begin{equation}
\frac1{\nu_1(\mu)}=\lim_{\eps\to0^+}\max\;\Spec\left(\sqrt{V_\mu}\frac1{D_0+1-\eps}\sqrt{V_\mu}\right).
\label{eq:lim_eps}
\end{equation}
If $t$ reaches $1/\nu_{\rm max}(\mu)$ before the eigenvalue touches $-1$, then necessarily
\begin{equation}
\lim_{\eps\to0^+}\max\;\Spec\left(\sqrt{V_\mu}\frac1{D_0+1-\eps}\sqrt{V_\mu}\right)\leq \nu_{\rm max}(\mu).
\label{eq:lim_eps2}
\end{equation}
but since the other inequality holds by~\eqref{eq:lower_bound_max_Spec} we see that there must be equality. We thus conclude that~\eqref{eq:lim_eps} holds for every probability measure $\mu$. 

Using Proposition~\ref{prop:nu_1} to estimate the right side of ~\eqref{eq:lim_eps} we obtain the upper bound 
\begin{equation}
\frac1{\nu_1(\mu)}\leq  \norm{\sqrt{V_\mu}\frac{1}{\alpha\cdot p}\sqrt{V_\mu}}_{L^2(\R^3,\C^4)\to L^2(\R^3,\C^4)}.
\label{eq:estim_nu_1_mu}
\end{equation}
On the other hand, we have 
$$\pscal{\sqrt{V_\mu}\Psi,\frac1{D_0+1-\eps}\sqrt{V_\mu}\Psi}\underset{\eps\to0^+}\longrightarrow \pscal{\sqrt{V_\mu}\Psi,\frac1{D_0+1}\sqrt{V_\mu}\Psi}$$
for any $\Psi$ in the form domain of $\sqrt{V_\mu}(D_0+1)^{-1}\sqrt{V_\mu}$, as discussed in Remark~\ref{rmk:form_domain}. By the same argument as for~\eqref{eq:argument_norm_wlsc} this implies again that 
$$\lim_{\eps\to0^+}\max\;\Spec\left(\sqrt{V_\mu}\frac1{D_0+1-\eps}\sqrt{V_\mu}\right)\geq \max\;\Spec\left(\sqrt{V_\mu}\frac1{D_0+1}\sqrt{V_\mu}\right).$$
Therefore  we have shown that 
$$\max\;\Spec\left(\sqrt{V_\mu}\frac1{D_0+1}\sqrt{V_\mu}\right)\leq \frac1{\nu_1(\mu)}\leq  \norm{\sqrt{V_\mu}\frac{1}{\alpha\cdot p}\sqrt{V_\mu}}_{L^2(\R^3,\C^4)\to L^2(\R^3,\C^4)}$$
for every probability measure $\mu$. After maximizing over $\mu$ we obtain 
\begin{multline}
\sup_{\substack{\mu\geq0\\ \mu(\R^3)=1}} \max\;\Spec\left(\sqrt{V_\mu}\frac{1}{D_0+1}\sqrt{V_\mu}\right)\\
\leq \frac1{\nu_1}\leq \sup_{\substack{\mu\geq0\\ \mu(\R^3)=1}} \norm{\sqrt{V_\mu}\frac{1}{\alpha\cdot p}\sqrt{V_\mu}}_{L^2(\R^3,\C^4)\to L^2(\R^3,\C^4)}.
\label{eq:lb}
\end{multline}
To show that there is equality in~\eqref{eq:lb}, we write 
\begin{equation*}
\sqrt{V_\mu}\frac{1}{\alpha\cdot p+\eps(\beta+1)}\sqrt{V_\mu}\\
=\sqrt{V_\mu}\frac{1}{\alpha\cdot p}\sqrt{V_\mu}+\eps(\beta-1)\sqrt{V_\mu}\frac{1}{|p|^2}\sqrt{V_\mu}. 
\end{equation*}
Next, we use that if $A=A^*$ and $B=B^*$ are two self-adjoint operators with $B$ bounded and $A=A^*$ is non-negative (but possibly unbounded), then 
\begin{equation}
\lim_{\eps\to0^+}\max\;\Spec(B-\eps A)=\max\;\Spec(B).
\label{eq:general_CV_A_B}
\end{equation}
The upper bound is obvious since $B-\eps A\leq B$ whereas the lower bound is obtained from the variational characterization~\eqref{eq:variational_max_spec} of the maximum of the spectrum. Namely, by density of $D(A)$ for any $\eta>0$ we can find a normalized vector $v\in D(A)$ so that $\pscal{v,Bv}\geq \max\Spec(B)-\eta$ and then 
$$\liminf_{\eps\to0^+}\max\Spec(B-\eps A)\geq \lim_{\eps\to0^+}\left(\pscal{v,Bv}-\eps\pscal{v,Av}\right)\geq \max\Spec(B)-\eta.$$
The claim follows after taking $\eta\to0$. We therefore obtain that 
\begin{multline*}
\lim_{\eps\to0^+}\max\;\Spec\left(\sqrt{V_\mu}\frac{1}{\alpha\cdot p+\eps(\beta+1)}\sqrt{V_\mu}\right)\\
=\max\;\Spec\left(\sqrt{V_\mu}\frac{1}{\alpha\cdot p}\sqrt{V_\mu}\right)=\norm{\sqrt{V_\mu}\frac{1}{\alpha\cdot p}\sqrt{V_\mu}}
\end{multline*}
where the last equality follows from the symmetry of the spectrum. But the left side is unitarily equivalent to $\sqrt{V_{\mu_\eps}}(D_0+1)^{-1}\sqrt{V_{\mu_\eps}}$ with $\mu_\eps=\eps^{-3}\mu(\eps^{-1}\cdot)$ and hence
$$\norm{\sqrt{V_\mu}\frac{1}{\alpha\cdot p}\sqrt{V_\mu}}\leq \sup_{\substack{\mu'\geq0\\ \mu'(\R^3)=1}} \max\;\Spec\left(\sqrt{V_{\mu'}}\frac{1}{D_0+1}\sqrt{V_{\mu'}}\right).$$
This shows that there are only equalities in~\eqref{eq:lb} and proves the second and third equalities in~\eqref{eq:formula_nu_1}. To conclude, it remains to notice that 
$$\frac1{\alpha\cdot p}=\begin{pmatrix}0&\frac1{\sigma\cdot p}\\\frac1{\sigma\cdot p}&0\\\end{pmatrix}$$
hence the norm of the operator in $L^2(\R^3,\C^4)$ is the same as the one of the off-diagonal term in $L^2(\R^3,\C^2)$:
$$\norm{\sqrt{V_\mu}\frac{1}{\alpha\cdot p}\sqrt{V_\mu}}_{L^2(\R^3,\C^4)\to L^2(\R^3,\C^4)}=\norm{\sqrt{V_\mu}\frac{1}{\sigma\cdot p}\sqrt{V_\mu}}_{L^2(\R^3,\C^2)\to L^2(\R^3,\C^2)}.$$
This provides the first equality in~\eqref{eq:formula_nu_1}, hence concludes the proof of Theorem~\ref{thm:critical_nu's}.\qed

\section{Existence of an optimal measure}\label{sec:proof_existence}

Our goal in this section is to prove Theorem~\ref{thm:existence_optimal_measure} on the existence of an optimal measure $\mu_\nu$ for $\lambda_1( \nu)$ in~\eqref{eq:min_lambda_1} and its properties. The difficulty is that the first eigenvalue which we minimize depends in a non trivial way on the measure $\mu$.  We will use the concentration-compactness method~\cite{Lions-84,Lions-84b,Lions-85a,Lions-85b,Struwe,Lewin-conc-comp} in order to prove that a minimizing sequence $\{\mu_n\}$ is necessarily tight and converges in the sense of measures (up to translations and extraction of a subsequence) to a minimizer $\mu_\nu$. This requires to eliminate the possible lack of compactness due to the space translations. The main idea of the proof is that \emph{local} compactness follows from the condition that the largest considered mass $\nu$ is strictly below $\nu_1$, which prevents the eigenvalue from approaching the lower continuous spectrum. On the other hand, if we think that the minimizing sequence $\mu_n$ splits into several compact bubbles far away from each other, the eigenvalue $\lambda_1(D_0-\mu_n\ast|x|^{-1})$ will be given by the \emph{smallest eigenvalue} of these bubbles. Should there be several bubbles, the mass of the bubble giving the lowest eigenvalue has to be strictly less than $\nu$. But then it cannot be optimal since we always do better by increasing it a bit and this is how we will prove that a minimizing sequence is necessarily tight, up to translations. 

In order to facilitate the writing of our proof, we will start by proving in Section~\ref{sec:vague_cont} that $\mu\mapsto \lambda_1(D_0-\mu\ast|x|^{-1})$ is weakly-$\ast$ continuous for the vague topology, that is, the one in duality with continuous functions tending to zero at infinity. This is one main ingredient for the proof of the existence in Theorem~\ref{thm:existence_optimal_measure}, which is then provided in Section~\ref{sec:proof_thm_existence}. Another important ingredient will be a unique continuation principle for the Dirac operator, the details of which are however gathered in Appendix~\ref{app:unique_continuation} for the convenience of the reader. 

\subsection{Continuity of the first eigenvalue for the vague topology}\label{sec:vague_cont}

In this section we will prove the weak continuity of the map $\mu\mapsto \lambda_1(D_0-\mu\ast|x|^{-1})$, using the concentration-compactness method. But before we provide the detailed statement, we prove two preliminary results which are going to be useful in the argument. The first will be used to deal with the case of \emph{vanishing sequences} $\{\mu_n\}$ which contain no compact bubble at all and have all of their mass disappearing locally. 

\begin{lemma}[Estimate in terms of the largest local mass]\label{lem:vanishing}
Let $\mu$ be a non-negative finite measure over $\R^3$. Then there exists a universal constant $C$ such that 
\begin{equation}
\norm{V_\mu \frac{1}{D_0}}\leq C\sup_{x\in\R^3}\mu(B_R(x))+\frac{C\mu(\R^3)}{R}
\end{equation}
for all $R\geq1$, where $B_R(x)$ is the ball of radius $R$ centered at $x\in\R^3$. 
\end{lemma}

\begin{proof}[Proof of Lemma~\ref{lem:vanishing}]
Let us consider a partition of unity $\sum_{j\in\Z^3}\chi_j=1$ of $\R^3$ with each $\chi_j\in C^\ii_c(\R^3)$ supported over the cube $j+(-1,1)^3$, for instance $\chi_j=\1_{(-1/2,1/2)^3}\ast\zeta(x-j)$ for a given $\zeta\in C^\ii_c(\R^3)$ with $\int_{\R^3}\zeta=1$, of small support. Let $\chi_{R,j}(x):=\chi_j(x/R)$ be the dilated partition of unity. Arguing as in the proof of~\cite[Lem.~8]{EstLewSer-20a_ppt}, we write 
\begin{align*}
\chi_{R,j}V_\mu&=\chi_{R,j}V_{\mu\1_{B_{4R}(Rj)}}+\chi_{R,j}V_{\mu\1_{\R^3\setminus B_{4R}(Rj)}}\\
&\leq \chi_{R,j}V_{\mu\1_{B_{4R}(Rj)}}+\frac{C\mu(\R^3)}{R}\chi_{R,j} 
\end{align*}
where $C$ depends on the smallest distance between the points on the sphere of radius $4$ and that on the cube of side length $2$. This gives
$$0\leq V_\mu\leq \sum_{j\in\Z^3}\chi_{R,j}V_{\mu\1_{B_{4R}(Rj)}}+\frac{C\mu(\R^3)}{R} $$
hence
$$\norm{V_\mu\frac1{\sqrt{1-\Delta}}}\leq \norm{\sum_{j\in\Z^3}\chi_{R,j}V_{\mu\1_{B_{4R}(Rj)}}\frac1{\sqrt{1-\Delta}}}+\frac{C\mu(\R^3)}{R}.$$
To estimate the first norm we write
\begin{multline*}
\sum_{j\in\Z^3}\chi_{R,j}V_{\mu\1_{B_{4R}(Rj)}}\frac1{\sqrt{1-\Delta}}\\=\sum_{j\in\Z^3}\chi_{R,j}V_{\mu\1_{B_{4R}(Rj)}}\frac1{\sqrt{1-\Delta}}\left(\1_{B_{4R}(Rj)}+\1_{\R^3\setminus B_{4R}(Rj)}\right) 
\end{multline*}
and estimate the corresponding positive kernels pointwise. Using that 
$$\frac1{\sqrt{1-\Delta}}(x-y)\leq C\frac{e^{-|x-y|}}{|x-y|^2}$$
we obtain
\begin{multline*}
\sum_{j\in\Z^3}\chi_{R,j}V_{\mu\1_{B_{4R}(Rj)}}\frac1{\sqrt{1-\Delta}}(x,y)\\
\leq \sum_{j\in\Z^3}\1_{B_{4R}(Rj)}(x)\frac{V_{\mu\1_{B_{4R}(Rj)}}(x)}{|x-y|^2}\1_{B_{4R}(Rj)}(y)
+\frac{C\mu(\R^3)}{R^3} e^{-|x-y|}.
\end{multline*}
By Hardy's inequality and the fact that $\sum_{j\in\Z^3}\1_{B_{4R}(Rj)}\leq C$, this proves that 
$$\norm{\sum_{j\in\Z^3}\chi_{R,j}V_{\mu\1_{B_{4R}(Rj)}}\frac1{\sqrt{1-\Delta}}(x,y)}\leq C\sup_{x\in\R^3}\mu\big(B_{4R}(x)\big)+\frac{C\mu(\R^3)}{R^3}$$
and concludes the proof.
\end{proof}

The second lemma is about the convergence of the Coulomb potential $V_{\mu_n}$ in the case where $\mu_n$ converges tightly or vaguely to a limit. 

\begin{lemma}[Convergence of the potential]\label{lem:CV_V_n}
Let $\mu_n\wto\mu$ be a sequence of measures which converges tightly. Then the associated potential $V_{\mu_n}=\mu_n\ast|x|^{-1}$ converges to $V_\mu=\mu\ast|x|^{-1}$ strongly in $(L^2+L^\ii)(\R^3)$, hence also almost everywhere after extraction of a subsequence. In particular, we have the norm convergence
$$\sqrt{V_{\mu_n}}\frac{1}{D_0-\lambda}\longrightarrow \sqrt{V_{\mu}}\frac{1}{D_0-\lambda}$$
for every $\lambda\in(-1,1)$, uniformly on compact subsets of $(-1,1)$. 

If $\mu_n\wto\mu$ converges vaguely (but not tightly), then we still have $V_{\mu_n}(x)\to V_\mu(x)$ strongly in $L^2_{\rm loc}(\R^3)$, hence also almost-everywhere after extraction of a subsequence.
\end{lemma}

\begin{proof}[Proof of Lemma~\ref{lem:CV_V_n}]
The tight convergence $\mu_n\wto\mu$ implies that the Fourier transforms $\widehat{\mu_n}(k)\to \widehat{\mu}(k)$ converge for all $k\in\R^3$. The Fourier transform of the corresponding potential can be written in the form
\begin{align*}
\widehat{V_{\mu_n}}(k)-\widehat{V_{\mu}}(k)&=4\pi\frac{\widehat{\mu_n}(k)-\widehat{\mu}(k)}{|k|^2}\\
&=4\pi\frac{\widehat{\mu_n}(k)-\widehat{\mu}(k)}{|k|^2}\1_{B_1}(k)+4\pi\frac{\widehat{\mu_n}(k)-\widehat{\mu}(k)}{|k|^2}\1_{\R^3\setminus B_1}(k) 
\end{align*}
where the first term is in $L^1(B_1)$ and the second in $L^2(\R^3\setminus B_1)$. From the dominated convergence theorem (using that $\widehat{\mu_n}$ is uniformly bounded) we infer that $V_{\mu_n}\to V_\mu$ strongly in $(L^2+L^\ii)(\R^3)$, hence in $L^2_{\rm loc}(\R^3)$. The last part of the statement follows from the inequality
\begin{multline*}
\norm{f(x)\frac{1}{D_0-\lambda}}\\
\leq \min\left(\frac{\norm{f}_{L^\ii}}{\min(|\lambda-1|,|\lambda+1|)},\frac{\norm{f}_{L^4}}{(2\pi)^{3}}\left(\int_{\R^3}\frac{dp}{|\alpha\cdot p+\beta-\lambda|^4}\,dp\right)^{\frac14}\right).
\end{multline*}
Finally, if we have $\mu_n\wto\mu$ vaguely (but not tightly), then we may always choose a radius $r_n$ diverging to infinity sufficiently slowly so that $\mu_n(B_{r_n})\to\mu(\R^3)$. Then $\mu_n\1_{B_{r_n}}$ converges tightly and on any fixed ball $B_R$ we have 
$$\1_{B_R}\left|V_{\mu_n}-V_{\mu_n\1_{B_{r_n}}}\right|=\1_{B_R}\left|V_{\mu_n\1_{\R^3\setminus B_{r_n}}}\right|\leq \frac{\mu_n(\R^3)}{r_n-R}\to0.$$
The local convergence therefore follows from the tight case.
\end{proof}

With these two results at hand, we will now be able to provide the proof of the following result, which is going to be one main ingredient for proving Theorem~\ref{thm:existence_optimal_measure}.

\begin{prop}[Weak continuity]\label{prop:weak_continuity}
Let $0\leq \nu<\nu_1$ and $\{\mu_n\}$ be an arbitrary sequence of non-negative measures such that $\mu_n(\R^3)\leq\nu$. Then there exists a subsequence $\{\mu_{n_k}\}$, a sequence of space translations $\{x_k\}\subset\R^3$ and a measure $\mu$ so that $\mu_{n_k}(\cdot+x_k)\wto \mu$ vaguely and 
$$\lim_{k\to\ii}\lambda_1\left(D_0-\mu_{n_k}\ast\frac1{|x|}\right)=\lambda_1\left(D_0-\mu\ast\frac1{|x|}\right).$$
\end{prop}

As we have mentioned above, the proof is a rather delicate application of Lions' concentration-compactness method~\cite{Lions-84,Lions-84b,Lions-85a,Lions-85b,Struwe,Lewin-conc-comp}. It will also largely rely on the min-max characterization of the eigenvalue recalled above in~\eqref{eq:min-max}. The argument thus has a strong variational flavor and does not solely rely on spectral arguments. Passing to the limit requires to prove first that the eigenvalue does not approach the bottom of the gap, that is,
$$\liminf_{n\to\ii}\lambda_1(D_0-\mu_n\ast|x|^{-1})>-1.$$
This is where we use that $\mu_n(\R^3)\leq\nu<\nu_1$ and this is the most difficult part of the proof. 

\begin{proof}[Proof of Proposition~\ref{prop:weak_continuity}]
First we notice that 
$$\lim_{n\to\ii}\lambda_1\left(D_0-V_{\mu\ast\zeta_n}\right)=\lambda_1(D_0-V_\mu)$$
for any regularizing sequence $\zeta_n\in C^\ii_c(\R^3)$ and any $0\leq\mu(\R^3)<1$. This follows from the resolvent convergence in~\cite[Thm.~1]{EstLewSer-20a_ppt}. Hence, we can always replace the original sequence $\mu_n$ by a regularized sequence, without changing the limit of the associated eigenvalue, nor the weak-$\ast$ limits of $\mu_n$. In the whole proof we assume for simplicity that $\mu_n\in C^\ii(\R^3,\R_+)$. This ensures that the domain of the corresponding Dirac operator is $H^1(\R^3)$ and allows us to carry some computations more easily. But the arguments below actually apply the same to a general measure. For the rest of the proof, we also call
$$\ell:=\lim_{n\to\ii}\lambda_1(D_0-V_{\mu_n})$$
the limit of the eigenvalues, which always exists after extraction of an appropriate (not displayed) subsequence. We split the proof of the proposition into several steps. In the first step we deal with the easy case where $\ell=1$, which contains in particular the case of ``vanishing'' sequences in the sense of concentration-compactness, as we will explain. The central argument is in Step 2 where we prove that $\ell>-1$, that is, the eigenvalue cannot approach the lower essential spectrum. In Steps 3 and 4 we will find the bubble $\mu$ which has the lowest possible eigenvalue and show that this is the limit of $\lambda_1(D_0-\mu_n\ast|x|^{-1})$. 

\medskip

\noindent $\bullet$ \textit{Step 1. First simple cases.}
If $\ell=1$ we can always find a sequence $\{x_n\}$ diverging fast enough to infinity so that $\mu_n(\cdot+x_n)\wto0=:\mu$. Since $\lambda_1(D_0-V_0)=\lambda_1(D_0)=1$, we see that the proposition is proved with $\mu=0$ in this case. In the rest of the proof we therefore assume that 
$$\ell<1.$$ 

Under this new assumption, $\mu_n(\R^3)$ cannot have a subsequence tending to 0. Otherwise using 
$$\norm{V_{\mu_{n_k}}\frac1{D_0}}\leq 2\mu_{n_k}(\R^3)\to0$$
by Hardy's inequality this would imply by the Rellich-Kato theorem that $\lambda_1(D_0-V_{n_k})\to1$. Thus we have
$$\liminf_{n\to\ii}\mu_n(\R^3)>0.$$

In fact, the sequence $\mu_n$ cannot vanish, in the sense of concentration-compactness~\cite{Lions-84,Lions-84b,Lions-85a,Lions-85b,Struwe,Lewin-conc-comp}. We recall that vanishing means that there is no asymptotic mass locally, in any ball of fixed radius $R>0$:
\begin{equation}
 \lim_{n\to\ii}\sup_{x\in\R^3}\mu_n(B_R(x))=0.
 \label{eq:mu_n_vanishes}
\end{equation}
The absence of vanishing follows from Lemma~\ref{lem:vanishing}. Indeed, when~\eqref{eq:mu_n_vanishes} we obtain $\|V_{\mu_n}D_0^{-1}\|\to0$ again and thus $\lambda_1(D_0-V_{\mu_n})\to1$, which contradicts our assumption that $\ell<1$. Thus we even have 
$$ \lim_{n\to\ii}\sup_{x\in\R^3}\mu_n(B_R(x))>0$$
for some (hence all) $R>0$. 

\medskip

\noindent $\bullet$ \textit{Step 2. Proof that $\ell>-1$.} Now we enter the core of the proof, which consists in showing that the eigenvalues $\lambda_1(D_0-V_{\mu_n})$ can never approach the bottom $-1$ of the gap, that is, $\ell>-1$. This is where we are going to use that $\mu_n(\R^3)\leq\nu<\nu_1$. We \textbf{argue by contradiction} and assume in this step that $\ell=-1$.
Denote by
\begin{equation}
M:=\sup\left\{\mu(\R^3)\ :\ \exists (x_k)\subset\R^3,\ \mu_{n_k}(\cdot-x_k)\wto\mu\text{ vaguely}\right\}
\label{eq:highest_mass}
\end{equation}
the largest mass of all the possible weak-$\ast$ vague limits of $\mu_n$ (when tested against continuous function tending to 0 at infinity), up to translations and extraction of a subsequence. If $M=0$ then we have $\mu_n(\cdot-x_n)\wto0$ for any $(x_n)\subset\R^3$ and this implies that $\mu_n(B_R(x_n))\to0$ for every $R>0$. This cannot happen due to the previous step. Therefore we must have $M>0$ and there exists a sequence of translations $(x_k)$ and a subsequence such that $\mu_{n_k}(\cdot-x_k)\wto\mu\neq0$ vaguely with, for instance, $\mu(\R^3)\geq M/2$. The problem being translation-invariant, we may assume for simplicity of notation that $x_k\equiv0$ and that $\mu_n\wto \mu$ vaguely, after extraction of a (not displayed) subsequence. Next, we denote for shortness 
$$\lambda_1(D_0-V_n)=:\lambda_n=-1+\eps_n,\qquad V_n:=\mu_n\ast\frac1{|x|}$$
where $\eps_n\to0^+$. We call $\Psi_n\in H^1(\R^3,\C^4)$ an eigenvector solving 
$$(D_0-V_n)\Psi_n=\lambda_n\Psi_n,\qquad \Psi_n=\begin{pmatrix}
\phi_n\\ \chi_n\end{pmatrix}.$$
We recall that the associated upper spinor $\phi_n$ is the \emph{first} solution of the eigenvalue equation
\begin{equation}
\left(-\sigma\cdot\nabla \frac{1}{\eps_n+V_{n}}\sigma\cdot\nabla +2-\eps_n-V_n\right)\phi_n=0
\label{eq:eigenvalue_eq_phi}
\end{equation}
and that 
$$\chi_n=\frac{-i\sigma\cdot\nabla \phi_n}{\eps_n+V_n}.$$
This follows from the min-max characterization of $\lambda_1$ as recalled in Section~\ref{sec:rappels} and proved in~\cite{EstLewSer-20a_ppt}. The quadratic form associated with the operator in~\eqref{eq:eigenvalue_eq_phi} is
$$q_{\lambda_n}(\phi):=\int_{\R^3}\frac{|\sigma\cdot\nabla\phi(x)|^2}{\eps_n+V_n(x)}\,dx+\int_{\R^3}\big(2-\eps_n-V_n(x)\big)|\phi(x)|^2\,dx\geq0.$$
It is non-negative because $\phi_n$ is the first eigenvector. 
In the whole argument we normalize our solution such that the upper spinor is itself normalized in $L^2$:
$$\boxed{\int_{\R^3}|\phi_n(x)|^2\,dx=1.}$$
With this choice we have little information on $\chi_n$, but this is the proper setting for using the min-max characterization of $\lambda_n=\lambda_1(D_0-V_n)$ and the quadratic form $q_{\lambda_n}$. Our strategy is to get some local compactness on $\phi_n$. Recall that the domain of the limiting operator $D_0-\mu\ast|x|^{-1}$ is included in $H^{1/2}(\R^3,\C^4)$. This suggests to show that $\{\phi_n\}$ is bounded in $H^{1/2}_{\rm loc}(\R^3,\C^2)$. 

Let $0\leq\zeta\leq 1$ be a smooth function with support in the ball $B_4$, which equals one in $B_2$ and set $\zeta_R(x):=\zeta(x/R)$ as well as $\eta_R=\sqrt{1-\zeta_R^2}$. We will use the \emph{pointwise} IMS formula for the Pauli operator which states that 
\begin{equation}
\sum_k|\sigma\cdot\nabla (J_k\phi)|^2=|\sigma\cdot \nabla\phi|^2+|\phi|^2\sum_k |\nabla J_k|^2,
\label{eq:IMS}
\end{equation}
for a partition of unity $\sum_k J_k^2=1$, see~\cite{EstLewSer-20a_ppt}. We obtain
\begin{align}
0&=q_{\lambda_n}(\phi_n)\nn\\
&=q_{\lambda_n}(\zeta_R\phi_n)+q_{\lambda_n}(\eta_R\phi_n)\nn\\
&\qquad -\int_{\R^3}\frac{|\sigma\cdot\nabla\zeta_R(x)|^2+|\sigma\cdot\nabla\eta_R(x)|^2}{\eps_n+V_n(x)}|\phi_n(x)|^2\,dx\nn\\
&\geq q_{\lambda_n}(\zeta_R\phi_n)+q_{\lambda_n}(\eta_R\phi_n)-\frac{C}{R^2}\int_{2R\leq |x|\leq 4R}\frac{|\phi_n(x)|^2}{\eps_n+V_n(x)}\,dx.\label{eq:localizing_f_q}
\end{align}
On the annulus $B_{4R}\setminus B_{2R}$ we have 
$$V_n(x)\geq \left(\mu_n\1_{B_R}\right)\ast\frac1{|x|}\geq \frac{\mu_n(B_R)}{5R}$$
where $5R$ is the largest possible distance between the points in the annulus and the points in the ball $B_R$. Since $\mu_n(B_R)\to\mu(B_R)$ with $\mu(\R^3)\geq M/2>0$ due to the vague convergence, we deduce that for $R$ large enough we have
\begin{equation}
q_{\lambda_n}(\zeta_R\phi_n)+q_{\lambda_n}(\eta_R\phi_n)\leq \frac{C}{R}\int_{B_{4R}\setminus B_{2R}}|\phi_n|^2\leq \frac{C}{R}.
\label{eq:bound_localized}
\end{equation}
Recall that $q_{\lambda_n}\geq0$, hence this gives a bound on  $q_{\lambda_n}(\zeta_R\phi_n)$ and $q_{\lambda_n}(\eta_R\phi_n)$ separately. 

We first look at the local part $q_{\lambda_n}(\zeta_R\phi_n)$ in~\eqref{eq:bound_localized} which gives, after discarding the $L^2$ term, 
\begin{equation}
\int_{\R^3}\frac{|\sigma\cdot \nabla\zeta_R \phi_n|^2}{\eps_n+V_n}\,dx-\int_{\R^3}V_n|\zeta_R\phi_n|^2\,dx\leq \frac{C}{R}.
\label{eq:second_local_term}
\end{equation}
For the second term in~\eqref{eq:second_local_term}, we use the characterization of $\nu_1$ in terms of the Hardy-type inequality~\eqref{eq:nu_c_Hardy_2component}, to infer
\begin{equation}
\int_{\R^3}V_n|\zeta_R\phi_n|^2\,dx\leq \frac{\mu_n(\R^3)^2}{\nu_1^2}\int_{\R^3}\frac{|\sigma\cdot \nabla\zeta_R \phi_n|^2}{V_n}\,dx.
\label{eq:estim_Coulomb_nu_1}
\end{equation}
For the first term in~\eqref{eq:second_local_term}, we use the lower bound
$$\frac1{\eps_n+V_n}=\frac1{V_n}-\frac{\eps_n}{(V_n+\eps_n)V_n}\geq \frac1{V_n}-C\frac{\eps_nR}{\eps_n+V_n}$$
where in the last estimate we have used that 
$$V_n\geq \frac{\mu_n(B_{4R})}{8R}\quad\text{on $B_{4R}$}.$$
We arrive at
$$\int_{\R^3}\frac{|\sigma\cdot \nabla\zeta_R \phi_n|^2}{\eps_n+V_n}\,dx\geq \frac1{1+C\eps_nR}\int_{\R^3}\frac{|\sigma\cdot \nabla\zeta_R \phi_n|^2}{V_n}\,dx.$$
We have therefore proved the following bound
\begin{equation}
\left(\frac1{1+C\eps_nR}-\frac{\mu_n(\R^3)^2}{\nu_1^2}\right)\int_{\R^3}\frac{|\sigma\cdot \nabla\zeta_R \phi_n|^2}{V_n}\,dx
\leq \frac{C}{R}.
\label{eq:estim_kinetic_phi_n}
\end{equation}
Since $\mu_n(\R^3)\leq \nu<\nu_1$ and $\eps_n\to0$ by assumption, this shows that the integral on the left side is uniformly bounded for fixed $R$. In particular, $\zeta_R\phi_n$ is also uniformly bounded in $\cV_{\mu_n}$, the space defined in~\eqref{eq:cV_mu_maximal}. Next, we show how this gives an $H^{1/2}$ bound. In~\cite[Thm.~2]{EstLewSer-20a_ppt} we have shown the inequality
\begin{equation}
\frac{\norm{\phi}^2_{H^{1/2}(\R^3,\C^2)}}{\max\big(2,16\,m(\R^3)\big)}\leq \norm{\phi}^2_{\cV_m}\leq \norm{\phi}^2_{H^{1}(\R^3,\C^2)}.
\label{eq:in_H_1/2}
\end{equation}
Scaling both $\phi$ and $m$ in~\eqref{eq:in_H_1/2}, we obtain the inequality
\begin{equation}
\int_{\R^3}\frac{|\sigma\cdot\nabla\phi(x)|^2}{\eta+V_m(x)}\,dx\geq \frac{\pscal{\phi,|p|\phi}-\eta\norm{\phi}_{L^2(\R^3)}^2}{2\max\big(1,8m(\R^3)\big)},
\label{eq:lower_bound_cV2}
\end{equation}
for all $\phi\in H^1(\R^3,\C^2)$, all positive measure $m$ and all $\eta>0$. Taking then $\eta\to0$ gives 
\begin{equation}
\pscal{\phi,|p|\phi}\leq 16 \,m(\R^3)\int_{\R^3}\frac{|\sigma\cdot\nabla\phi(x)|^2}{V_m(x)}\,dx.
\label{eq:lower_bound_cV_eps0}
\end{equation}
Using this inequality in~\eqref{eq:estim_kinetic_phi_n} gives
\begin{equation}
\left(\frac1{1+C\eps_nR}-\frac{\mu_n(\R^3)^2}{\nu_1^2}\right)\pscal{\zeta_R\chi_n,|p|\zeta_R\phi_n}\leq \frac{C}{R}.
\label{eq:estim_H12_phi_n}
\end{equation}
This shows that $\zeta_R\phi_n$ is bounded in $H^{1/2}$ for every $R$ large enough. In other words, $\phi_n$ is bounded in $H^{1/2}_{\rm loc}$, as we claimed. 

After extraction of a subsequence, we may assume that $\phi_n\wto \phi$ weakly in $L^2$ and strongly in $L^2_{\rm loc}$, hence also almost everywhere. From Lemma~\ref{lem:CV_V_n} we also have that $V_n(x)\to V_\mu(x)$ almost-everywhere. Passing to the limit in~\eqref{eq:estim_Coulomb_nu_1} using~\eqref{eq:estim_kinetic_phi_n} we obtain from Fatou's lemma
$$\int_{\R^3}V_\mu|\zeta_R\phi|^2\,dx\leq \frac{C}{R}.$$
Taking finally $R\to\ii$ gives $\phi\equiv0$. 

Using the strong local compactness, we can choose $R=R_n\to\ii$ sufficiently slowly to ensure that 
$$\eps_n R_n\to0,\qquad \mu_n(B_{R_n})\to\mu(\R^3),\qquad \mu_n(B_{8R_n}\setminus B_{R_n})\to0,$$
$$\int_{B_{8R_n}}|\phi_n|^2\to0.$$
From~\eqref{eq:estim_H12_phi_n} we also have
$$\norm{\zeta_{R_n}\phi_n}_{H^{1/2}}\to0.$$
All this shows that nothing is happening in the region under investigation. The mass of $\phi_n$ must be at infinity. 

At this step we have decomposed our quadratic form as in~\eqref{eq:bound_localized} and have shown that $\phi_n$ has no $L^2$ mass in the region where $\mu_n$ converges to $\mu$. The next step is to apply the whole argument again to $\eta_{R_n}\phi_n$. Namely, we extract the next profile in the sequence $\mu_n$ and use the same argument to show that $\phi_n$ has no mass in the corresponding region. After finitely many steps the remainder $\mu'_n$ will be composed of a piece which can vanish and another piece with an arbitrarily small mass (for instance a mass $\leq1/2$). For simplicity of exposition, we provide the end of the argument in the simplest situation, namely we assume that 
$$\mu_n\1_{\R^3\setminus B_{R_n}}=\mu_n^{(1)}+\mu_n^{(2)}$$
where $\mu_n^{(1)}$ vanishes in the sense of~\eqref{eq:mu_n_vanishes} and $\mu_n^{(2)}(\R^3\setminus B_{R_n})\leq 1/2$. The argument in the general case is similar but more tedious to write down. 
By Lemma~\ref{lem:vanishing} and Hardy's inequality, this implies that 
$$\lambda_1(\nu_n,\mu_n\1_{\R^3\setminus B_{R_n}})\geq 0$$
for $n$ large enough. Hence, by the min-max principle and the characterization in terms of the quadratic form $q_\lambda$, this tells us that 
\begin{multline}
\int_{\R^3}\frac{|\sigma\cdot\nabla\phi|^2}{1+ V_{\mu_n\1_{\R^3\setminus B_{R_n}}}}\,dx+\int_{\R^3}\big(1-V_{\mu_n\1_{\R^3\setminus B_{R_n}}}\big)|\phi|^2\,dx\geq0,\\ \forall\phi\in H^1(\R^3,\C^2).
\label{eq:estim_quad_0}
\end{multline}
On the support of $\eta_{R_n}$ in~\eqref{eq:bound_localized} we have
$$V_n=V_{\mu_n\1_{\R^3\setminus B_{R_n}}}+ V_{\mu_n\1_{B_{R_n}}}\leq  V_{\mu_n\1_{\R^3\setminus B_{R_n}}}+\frac1{R_n}$$
hence we obtain from~\eqref{eq:estim_quad_0} and the fact that $\eps_n+R_n^{-1}\to0$
\begin{align}
q_{\lambda_n}(\eta_{R_n}\phi_n)&\geq \int_{\R^3}\frac{|\sigma\cdot\nabla(\eta_{R_n}\phi_n)|^2}{\eps_n+R_n^{-1}+ V_{\mu_n\1_{\R^3\setminus B_{R_n}}}}\,dx\nn\\
&\qquad+\int_{\R^3}\left(2-\eps_n-\frac1{R_n}-V_{\mu_n\1_{\R^3\setminus B_{R_n}}}\right)|\eta_{R_n}\phi_n|^2\,dx\nn\\
&\geq \int_{\R^3}\left(1-\eps_n-\frac1{R_n}\right)|\eta_{R_n}\phi_n|^2\,dx.\label{eq:small_mass_region}
\end{align}
Since the left side is $\leq C/R_n$ and we have already proved that $\zeta_{R_n}\phi_n$ tends to zero, this shows that $\phi_n\to0$ in $L^2$, a contradiction to its normalization. Hence we conclude that $\ell=-1$ cannot happen, as was claimed. 

\medskip

We have succeeded in showing that the eigenvalues $\lambda_1(D_0-V_{\mu_n})$ cannot approach $-1$. Our next goals are to 
\begin{itemize}
 \item[$(i)$] extract from $\mu_n$ one tight piece of mass $\widetilde \mu_n=\mu_n\1_{B_{R_n}(x_n)}\wto\mu\neq0$ vaguely for a proper space translation $\{x_n\}\subset\R^3$, such that the corresponding eigenvalue $\lambda_1(D_0-V_{\widetilde\mu_n})$ has the \emph{same limit $\ell$} as the original sequence~$\mu_n$;
\item[$(ii$)] prove that $\lambda_1(D_0-V_{\widetilde\mu_n})\to\lambda_1(D_0-V_{\mu})$, using the tightness of $\widetilde \mu_n$.
\end{itemize}
It is convenient to start with $(ii)$, that is, to show that when a sequence of measures converges tightly to a limit $\mu$ and has masses below $\nu_1$, then the first eigenvalue converges. In Step 4 we then explain how to prove $(i)$. 

\medskip

\noindent $\bullet$ \textit{Step 3. Convergence in the tight case.}
In this step we prove the weak continuity of $\mu\mapsto \lambda_1(D_0-V_\mu)$ for the \emph{tight} convergence of measures. Since this is a result of independent interest, we state it as an independent lemma. 

\begin{lemma}[Convergence in the tight case]\label{lem:weak_continuity}
Let $0\leq\nu<\nu_1$. Let $\{\mu_n\}$ be a sequence of non-negative measures such that $\mu_n(\R^3)\leq\nu$ and which converges tightly to a measure $\mu$. Then we have 
\begin{equation}
\lim_{n\to\ii}\lambda_1\left(D_0-\mu_n\ast\frac1{|x|}\right)=\lambda_1\left(D_0-\mu\ast\frac1{|x|}\right).
\end{equation}
\end{lemma}

\begin{proof}[Proof of Lemma~\ref{lem:weak_continuity}]
Since $\nu<\nu_1$, we can pick an $\eta>0$ such that 
$\nu(1+\eta)<\nu_1$
and consider the sequence $V'_n:=(1+\eta)\mu_n\ast|x|^{-1}=(1+\eta)V_n$ where $(1+\eta)\mu_n(\R^3)\leq(1+\eta)\nu<\nu_1$. The first part of the proof implies that there exists $\eps_0>0$ so that 
$$\lambda_1\Big(D_0-(1+\eta)V_n\Big)> -1+\eps_0$$
for $n$ large enough. From the Birman-Schwinger principle in~\cite[Thm.~3]{EstLewSer-20a_ppt} this is equivalent to saying that 
$$\max\Spec\left(\sqrt{V_n}\frac{1}{D_0+1-\eps_0}\sqrt{V_n}\right)<\frac1{1+\eta}<1.$$
Therefore, we have the operator bound
$$K_n:=\sqrt{V_n}\frac{1}{D_0+1-\eps_0}\sqrt{V_n}<\frac1{1+\eta}.$$
By Lemma~\ref{lem:CV_V_n} we have the strong convergence
$$K_n\to K=\sqrt{V}\frac{1}{D_0+1-\eps_0}\sqrt{V}$$
and the uniform upper bound implies from the functional calculus that 
$$(1-K_n)^{-1}\to(1-K)^{-1}$$
strongly as well. In addition Lemma~\ref{lem:CV_V_n} provides the norm convergence of $\sqrt{V_n}(D_0+1-\eps_0)^{-1}$. From the resolvent formula~\cite{Nenciu-76,KlaWus-79,Klaus-80b} 
\begin{multline}
\frac1{D_0-V-E}=\frac1{D_0-E}\\-\frac1{D_0-E}\sqrt{V_n}\frac1{1-\sqrt{V_n}\frac1{D_0-E}\sqrt{V_n}}\sqrt{V_n}\frac1{D_0-E} 
 \label{eq:resolvent}
\end{multline}
with $E=-1+\eps_0$, we conclude that 
$$(D_0-V_n+1-\eps_0)^{-1}\to(D_0-V+1-\eps_0)^{-1}$$
in norm as $n\to\ii$. The convergence of the resolvent implies the convergence of the spectrum. In particular, the first eigenvalue $\lambda_1(D_0-V_n)$ (which is known to be larger than $-1+\eps_0$ by the above arguments) converges to $\lambda_1(D_0-\mu\ast|x|^{-1})$ and this concludes the proof of Lemma~\ref{lem:weak_continuity}.
\end{proof}

\medskip

\noindent $\bullet$ \textit{Step 4. Extraction of a tight minimizing sequence.}
Next, we go back to our initial minimizing sequence $\mu_n$, for which we know that $-1<\ell<1$. We would like to extract from $\mu_n$ a new \emph{tight} minimizing sequence, by removing the unnecessary parts going to infinity. 

The idea is the following. We apply the same strategy as in the previous step and extract finitely many weak limits of $\mu_n$ up to translations, so that the remainder can be written in the form $\mu_n'=\mu_n^{(1)}+\mu_n^{(2)}$ where $\mu_n^{(1)}$ vanishes in the sense of~\eqref{eq:mu_n_vanishes} and $\mu_n^{(2)}(\R^3)\leq \eta\ll1$. This time we choose $\eta$ to guarantee that 
$\lambda_1\big(D_0-V_{\mu_n^{(2)}}\big)>\ell$. By an argument similar to the one in~\eqref{eq:small_mass_region}, we can prove that $\phi_n$ converges to $0$ in $L^2$ on the support of $\mu_n'$. Hence it must have a non zero mass in one of the regions where $\mu_n$ converges tightly to a non-zero measure. We then show that the eigenvalue of this particular tight piece converges to $\ell$. 

For the sake of clarity, we write again the whole argument in the simplest situation where we only have one tight piece. Thus we have like in the previous proof that $\mu_n\1_{B_{R_n}}\to\mu$ tightly, whereas $\mu_n':=\mu_n\1_{\R^3\setminus B_{R_n}}=\mu_n^{(1)}+\mu_n^{(2)}$ where $\mu_n^{(1)}$ vanishes and $\mu_n^{(2)}(\R^3\setminus B_{R_n})\leq \eta\ll1$. Then from~\eqref{eq:bound_localized} we know that $\eta_{R_n}\phi_n\to0$, which implies that 
$$\lim_{n\to\ii}\int_{\R^3}|\zeta_{R_n}\phi_n|^2=1.$$
We have in addition
\begin{equation}
q_{\lambda_n}(\zeta_{R_n}\phi_n)\leq \frac{C}{R_n} \leq \frac{C'}{R_n}\int_{\R^3}|\zeta_{R_n}\phi_n|^2
\label{eq:bound_localized2}
\end{equation}
since the last integral converges to 1. On the support of $\zeta_{R_n}$ we have as before
$$V_n\leq \nu_n V_{\mu_n\1_{B_{8R_n}}}+\frac{C}{R_n}$$
hence we obtain
\begin{multline*}
0\geq  \int_{\R^3}\frac{|\sigma\cdot\nabla(\zeta_{R_n}\phi_n)|^2}{1+\lambda_n+C/R_n+ V_{\mu_n\1_{B_{8R_n}}}}\,dx\\
+\int_{\R^3}\left(1-\lambda_n-\frac{C}{R_n}-V_{\mu_n\1_{\R^3\setminus B_{8R_n}}}\right)|\zeta_{R_n}\phi_n|^2\,dx.
\end{multline*}
From the characterization of the first eigenvalue via the quadratic form, this proves that 
\begin{equation}
\lambda_1\left(D_0-V_{\mu_n\1_{B_{8R_n}}}\right)\leq \lambda_n+\frac{C}{R_n}.
\label{eq:upper_lambda_1_tight_case}
\end{equation}
From the convergence in Lemma~\ref{lem:weak_continuity} the left side converges to $\lambda_1(D_0-V_\mu)$ and hence we obtain after passing to the limit
$$\lambda_1(D_0-V_\mu)\leq\ell.$$
On the other hand, for every fixed $\phi\in C^\ii_c(\R^3,\C^2)$ we have $q_{\lambda_n,\mu_n}(\phi)\geq0$ and passing to the limit using the strong local convergence of $V_n$ from Lemma~\ref{lem:CV_V_n}, we obtain
$q_{\ell,\mu}(\phi)\geq0$. This precisely means that 
$$\ell\leq \lambda_1(D_0-V_\mu).$$
Thus we have proved, as desired, that $\ell=\lambda_1(D_0-V_\mu)$ and this concludes the proof of Proposition~\ref{prop:weak_continuity}.
\end{proof}

\begin{remark}
The function $\nu\mapsto \lambda_1(\nu)$ is non-increasing and positive for $\nu\leq2/(\pi/2+2/\pi)$ by Theorem~\ref{thm:critical_nu's}. Therefore, there exists a critical number $\nu_1'\geq 2/(\pi/2+2/\pi)$ for which $\lambda_1(\nu)>-1$ on $[0,\nu_1')$ and $\lambda_1(\nu)=-1$ on $(\nu_1',1)$. It is not difficult to see that $\nu_1'\leq\nu_1$, the critical number defined in~\eqref{eq:def_nu_c}. In Step 3 we proved that $\ell>-1$ for every sequence $\{\mu_n\}$ so that $\mu_n(\R^3)\leq \nu<\nu_1$ and this actually shows that $\nu_1'=\nu_1$. In other words, we can exchange the limit over $\nu$ with the infimum over $\mu$. 
\end{remark}

\subsection{Proof of Theorem~\ref{thm:existence_optimal_measure}}\label{sec:proof_thm_existence}

We first prove the existence of an optimal measure using Proposition~\ref{prop:weak_continuity}, before we discuss the other parts of the statement.

\medskip

\noindent $\bullet$ \textit{Step 1. Existence of an optimizer.} Let $\{\mu_n\}$ be any minimizing sequence for $\lambda_1(\nu)$, with $0<\nu<\nu_1$ and $\mu_n(\R^3)=\nu$. From Proposition~\ref{prop:weak_continuity} we know that there exists a subsequence and space translations $\{x_k\}\subset\R^3$ so that $\mu_{n_k}(\cdot+x_k)\wto\mu$ vaguely (hence $\mu(\R^3)\leq\nu$) and
$$\lambda_1(\nu)=\lim_{n\to\ii}\lambda_1(D_0-V_{\mu_n})=\lambda_1(D_0-V_\mu).$$
The measure $\mu$ is the desired optimizer. To prove that the convergence is in fact tight, we have to show that $\mu(\R^3)=\nu$. The argument here relies on the strict monotonicity of the eigenvalue, as we now explain. First, for $\nu>0$ we have 
$$\lambda_1(\nu)\leq\lambda_1(D_0-\nu/|x|)=\sqrt{1-\nu^2}<1,$$ 
we deduce that $\mu\neq0$ (that is, the sequence $\{\mu_n\}$ cannot vanish). On the other hand, if $\mu(\R^3)<\nu$ we have
\begin{equation}
\lambda_1(\nu)=\lambda_1(D_0-V_\mu)>\lambda_1\left(D_0-\frac{\nu}{\mu(\R^3)}V_\mu\right)\geq \lambda_1(\nu), 
 \label{eq:proof_total_mass}
\end{equation}
a contradiction. Hence $\mu(\R^3)=\nu$ and the original sequence must be tight. In~\eqref{eq:proof_total_mass} we have used that $t\mapsto \lambda_1(D_0-tV_\mu)$ is decreasing for a fixed $\mu$. This follows from the min-max principle and the characterization in terms of quadratic forms~\cite{DolEstSer-00}. Indeed, if $\phi_\nu\neq0$ is an eigenfunction associated with $\lambda_1(D_0-V_\mu)$, we have
\begin{multline*}
\int_{\R^3}\frac{|\sigma\cdot\nabla\phi_\nu|^2}{1+\lambda_1(D_0-V_\mu)+t V_{\mu}}\,dx+\int_{\R^3}\big(1-\lambda_1(D_0-V_\mu)-tV_{\mu}\big)|\phi_\nu|^2\,dx\\
<q_{\lambda_1(D_0-V_\mu)}(\phi_\nu)=0 
\end{multline*}
for $t>1$ since $V_\mu>0$ everywhere. This shows that $\lambda_1(D_0-tV_\mu)<\lambda_1(D_0-V_\mu)$ for $t>1$. 

\medskip

\noindent $\bullet$ \textit{Step 2. Properties of $\nu\mapsto \lambda_1(\nu)$.}
The function $\nu\mapsto \lambda_1(\nu)$ is known to be non-increasing for $\nu\in[0,\nu_1)$. Since there exists a minimizer $\mu$ for every $\nu$ the previous argument shows that $\nu\mapsto \lambda_1(\nu)$ is decreasing. Hence it is continuous except possibly on a countable set. 

To prove the continuity, consider a sequence $\nu_n\to\nu\in(0,\nu_1)$ together with an associated sequence of optimizers $\mu_n$ so that $\lambda_1(D_0-V_{\mu_n})=\lambda_1(\nu_n)$. From Proposition~\ref{prop:weak_continuity} we know that we can assume $\mu_n\wto \mu\neq0$ vaguely after an appropriate translation and extraction of a subsequence, so that 
$$\liminf_{n\to\ii}\lambda_1(D_0-V_{\mu_n})=\lim_{n\to\ii}\lambda_1(D_0-V_{\mu_n})=\lambda_1(D_0-V_{\mu})\geq\lambda_1(\nu).$$
Let $\mu$ be an optimizer for $\lambda_1(\nu)$. We use $(\nu_n/\nu) \mu$ as a trial state for $\lambda_1(\nu_n)$ and obtain 
$$\limsup_{n\to\ii}\lambda_1(\nu_n)\leq \lim_{n\to\ii}\lambda_1\left(D_0-\frac{\nu_n}{\nu}V_\mu\right)=\lambda_1(D_0-V_\mu)=\lambda_1(\nu)$$
since the map $t\mapsto \lambda_1(D_0-tV_\mu)$ is continuous for a fixed $\mu$. This concludes the proof of the continuity of $\nu\mapsto \lambda_1(\nu)$. 

Finally, we discuss the regularity of $\nu\mapsto\lambda_1(\nu)$. It is well known that for every fixed $\mu$, the function $t \mapsto \lambda_1(D_0-tV_\mu)$ is Lipschitz~\cite{Kato}. The Lipschitz constant is uniformly bounded, for $t$ in any compact set of $[0,\nu_1/\mu(\R^3))$. This follows from the resolvent formula
\begin{align*}
&\frac1{D_0-tV_\mu+1-\eps_0}-\frac1{D_0-t'V_\mu+1-\eps_0} \\
&\qquad =(t-t')\frac1{D_0-t V_\mu+1-\eps_0}V_\mu \frac1{D_0-t'V_\mu+1-\eps_0}
\end{align*}
which implies 
\begin{align*}
&\norm{\frac1{D_0-t V_\mu+1-\eps_0}-\frac1{D_0-t'V_\mu+1-\eps_0}} \\
&\qquad \leq \frac\pi2|t-t'| \norm{\frac1{D_0-t V_\mu+1-\eps_0}|D_0|^{\frac12}}\norm{|D_0|^{\frac12}\frac1{D_0-t'V_\mu+1-\eps_0}}.
\end{align*}
Here $\eps_0:=\lambda_1(\nu_1-\eta)+1>0$ where $\eta>0$ is chosen so that $t,t'<(\nu_1-2\eta)/\mu(\R^3)$. The two norms can be estimated uniformly in $\mu$ using the resolvent formula~\eqref{eq:resolvent} and the fact that 
$$\sqrt{V_\mu}\frac{1}{D_0+1-\eps_0}\sqrt{V_\mu}\leq\frac1{\nu_1-\eta}.$$
To see that the Lipschitz property at fixed $\mu$ implies a similar property for $\lambda_1(\nu)$, we remark that for $\nu\leq\nu'$
$$\lambda_1(\nu')\leq \lambda_1(\nu)\leq \lambda_1(\nu,\mu')=\lambda_1(\nu',\mu')+C(\nu'-\nu)$$
where $\mu'$ is a minimizer for $\lambda_1(\nu')$.

\medskip

\noindent $\bullet$ \textit{Step 3. Euler-Lagrange equation.}
Let $\mu$ be a minimizer for $\lambda_1(\nu)$ and let $\Psi=(\phi,\chi)$ be any corresponding eigenfunction. Recall that $\phi$ solves~\eqref{eq:eigenvalue_eq_phi} and that 
$$\chi=\frac{-i\sigma\cdot\nabla \phi}{1+\lambda_1(\nu)+ V_\mu}.$$
Let $\mu'$ be any other probability measure and $\mu_t:=(1-t)\mu+t\mu'$, for $t\in[0,1]$. Then we have $\lambda_1(D_0-V_{\mu_t})\geq\lambda_1(D_0-V_\mu)$ and this implies that $q_{\lambda_1,\mu_t}(\phi)\geq0$ for all $t\in[0,1]$. Expanding in $t$ gives that 
$$\int_{\R^3}\left(|\Psi|^2\ast\frac1{|\cdot|}\right)(x)\,{\rm d}(\mu'-\mu)(x)\leq0$$
for all such $t\in[0,1]$, where $|\Psi|^2=|\phi|^2+|\chi|^2$. In other words, $\mu$ solves the maximization problem
$$\sup_{\substack{\mu'\geq0\\ \mu'(\R^3)=1}}\int_{\R^3}\left(|\Psi|^2\ast\frac1{|\cdot|}\right)(x)\,{\rm d}\mu'(x).$$
Since $\Psi\in H^{1/2}(\R^3,\C^4)$ by~\cite[Thm.~1]{EstLewSer-20a_ppt}, the potential $|\Psi|^2\ast|x|^{-1}$ is actually a continuous function tending to zero at infinity. The solutions to the maximization problem are exactly the measures supported on the compact set where this function attains its maximum. In particular $\mu$ concentrates on the compact set
$$K:={\rm argmax}\left(|\Psi|^2\ast\frac1{|x|}\right).$$

\medskip

\noindent $\bullet$ \textit{Step 4. $K$ has zero measure.}
The final step is to prove that $K$ has zero Lebesgue measure. To this end, we argue by contradiction and show that when $|K|>0$ the corresponding $\Psi$ would vanish to all orders at one point of $K$. This is impossible by unique continuation, as explained in Appendix~\ref{app:unique_continuation}. 

Let us therefore assume that $|K|>0$ and denote by 
$$\Omega:=\R^3\setminus\{R_1,...,R_K\}$$
the set obtained after removing the largest singularities of $\mu$, for instance all the points so that $\mu(\{R_j\})\geq \min(1/4,\eps_0/4)$ where $\eps_0$ is the universal constant from Theorem~\ref{thm:unique_continuation} in Appendix~\ref{app:unique_continuation}. Then of course $|K\cap\Omega|>0$ as well. Let us denote by 
$$U:=\max_{\R^3}\left(|\Psi|^2\ast\frac1{|x|}\right)-|\Psi|^2\ast\frac1{|x|}\geq0$$
the shifted potential which satisfies $U\equiv0$ on $K$ as well as the equation
$$\Delta U=4\pi|\Psi|^2\geq0$$
on $\R^3$. Consider a point of full measure $x_0\in \Omega\cap K$, that is, such that 
$$\lim_{r\to0}\frac{|B_r(x_0)\setminus K|}{|B_r(x_0)|}=0.$$
To simplify our exposition we assume without loss of generality that $x_0=0$. 
Let $\chi\in C^\ii_c(B_2)$ be so that $\chi_{|B_1}\equiv1$ and set $\chi_r(x):=\chi(x/r)$. Then we have
$$-\chi_r U\Delta(\chi_r U)=-4\pi\chi_r^2U|\Psi|^2-\chi_r U^2\Delta\chi_r-\frac12\nabla\chi_r^2\cdot \nabla U^2.$$
The first term on the right side is non-positive since $U\geq0$. Integrating we obtain
$$\int_{\R^3}|\nabla (\chi_r U)|^2\leq  -\int_{\R^3}\chi_r U^2\Delta\chi_r+\frac12\int_{\R^3}U^2\Delta\chi_r^2=\int_{\R^3}U^2|\nabla \chi_r|^2$$
and therefore
\begin{align*}
\int_{B_r}\left(\frac{U^2}{r^2}+|\nabla U|^2\right)&\leq \int_{\R^3}\left(\frac{\chi_r^2U^2}{r^2}+|\nabla (\chi_rU)|^2\right)\\
&\leq  \frac{C}{r^2}\int_{B_{2r}}U^2=\frac{C}{r^2}\int_{B_{2r}\setminus K}U^2 
\end{align*}
since $U\equiv0$ on $K$ by definition. Next, we use the Sobolev inequality~\cite[Thm.~4.12]{AdamsFournier} in the ball $B_{2r}$
and obtain 
$$ \frac{C}{r^2}\int_{B_{2r}\setminus K}U^2\leq  \frac{C|B_{2r}\setminus K|^{\frac23}}{r^2}\norm{U}_{L^6(B_{2r})}^2\leq \frac{C|B_{2r}\setminus K|^{\frac23}}{r^2}\int_{B_{2r}}\left(\frac{U^2}{4r^2}+|\nabla U|^2\right)$$
(the dependence in $r$ follows by scaling). Hence in summary we have proved that
$$ \int_{B_r}\left(\frac{U^2}{r^2}+|\nabla U|^2\right)\leq C\left(\frac{|B_{2r}\setminus K|}{|B_r|}\right)^{\frac23}\int_{B_{2r}}\left(\frac{U^2}{4r^2}+|\nabla U|^2\right)$$
for a universal constant $C$. By arguing like in~\cite[Section~3]{FigGos-92} this proves that 
\begin{equation}
 \lim_{r\to0^+} r^{-\alpha}\int_{B_r}\left(\frac{U^2}{r^2}+|\nabla U|^2\right)=0,\qquad \forall\alpha>0,
 \label{eq:U_vanishes_any_order}
\end{equation}
that is, $U$ and $\nabla U$ vanish to all orders at $x_0=0$. 

Next, we prove that $\Psi$ also vanishes to all orders at the same point. We use Green's formula in the form
\begin{align*}
4\pi\int_{B_r}|\Psi|^2=\int_{B_r}\Delta U=-\int_{S_r}\nabla  U\cdot n
\end{align*}
where $n$ is the outward normal to the sphere $S_r$ of radius $r$. Note that $\Psi\in H^1(\Omega)$ by~\eqref{eq:domain_in_H1_loc} since in $\Omega$ we have removed the largest singularities. In particular, $\nabla U$ is indeed a continuous function on $\Omega$ by Hardy's inequality. Although we could show that $\nabla U$ vanishes to all orders when integrated on the sphere $S_r$, we prefer to bound it in terms of $U$. After passing to spherical coordinates we see that 
$$\int_{S_r}\nabla  U\cdot n=r^2\frac{d}{d r}\left(\frac1{r^2}\int_{S_r}U\right)$$
therefore we obtain after integrating over $r$
\begin{equation}
\frac{\pi}{r^2}\int_{B_r}|\Psi|^2\leq 4\pi\int_r^{2r}\int_{B_s}|\Psi|^2\,\frac{ds}{s^2}=\frac1{r^2}\int_{S_r}U-\frac1{4r^2}\int_{S_{2r}}U.
\label{eq:estim_Psi_U}
\end{equation}
Since $U\geq0$ we have shown the inequality
$$\pi\int_{B_r}|\Psi|^2\leq \int_{S_r}U.$$
From the continuity of boundary traces in the ball $B_r$ we have
$$\int_{S_r}U\leq Cr\int_{B_r}\left(\frac{U^2}{r^2}+|\nabla U|^2\right)$$
which vanishes to all orders, as we have shown in~\eqref{eq:U_vanishes_any_order}. By~\eqref{eq:estim_Psi_U} this proves that $\Psi$ also vanishes to all orders at the same point. This is impossible by Corollary~\ref{cor:unique_continuation_general_charge} in Appendix~\ref{app:unique_continuation}. Hence we must have $|K|=0$ and this concludes the proof of Theorem~\ref{thm:existence_optimal_measure}.\qed

\appendix
\section{A unique continuation principle for Dirac operators}\label{app:unique_continuation}

The unique continuation principle for Dirac operators has been the object of many works, including for instance~\cite{Vogelsang-87,Jerison-86,Wolff-90,Mandache-94,Kim-95,CarOka-99,KalYam-99,Booss-00}. Here we prove a result which we have not been able to locate in the literature and which is well adapted to the case of Coulomb potentials generated by an arbitrary charge charge distribution $\mu$. 

\begin{thm}[Strong unique continuation for Dirac operators]\label{thm:unique_continuation}
Let $\Omega\subset \R^3$ be a connected open set and $V\in L^{3,\ii}_{\rm loc}(\Omega,\R_+)$. Let $\Psi\in L_{\rm loc}^{3,\frac32}(\Omega,\C^4)$ be such that $\nabla \Psi\in L_{\rm loc}^{\frac32}(\Omega)$ and solving the differential inequality
\begin{equation}
|\alpha\cdot\nabla \Psi(x)|\leq V(x)|\Psi(x)|\qquad\text{on $\Omega$.}
\end{equation}
There exists a universal constant $\eps_0$ such that if
$$\limsup_{r\to 0}\norm{V}_{L^{3,\ii}(B_r(x_0))}\leq \eps_0$$
for every $x_0\in \Omega$, and if $\Psi$ vanishes to all orders at one point $x_1\in\Omega$, that is,
$$\lim_{r\to0^+}r^{-\alpha}\int_{B_r(x_1)}|\Psi|^2=0,\qquad\forall \alpha>0,$$
then $\Psi\equiv0$ everywhere on $\Omega$. 
\end{thm}

Here $L^{p,q}$ are Lorentz spaces~\cite{Grafakos-book}. The case where $V\in L^3_{\rm loc}(\R^3)$ is treated in~\cite{Wolff-90} but we are not aware of a result for potentials which are small locally in the Lorentz space $L^{3,\ii}$. For the Laplacian a result similar to Theorem~\ref{thm:unique_continuation} with $V$ small in $L^{3/2,\ii}$ was proved first by Stein in~\cite{Stein-85} based on ideas of Jerison-Kenig~\cite{JerKen-85}. This has recently been generalized to fractional Laplacians by Seo in~\cite{Seo-15}, a result on which we rely in our proof. 

\begin{proof}[Proof of Theorem~\ref{thm:unique_continuation}]
The proof is based on the following Carleman inequality in Lorentz spaces
\begin{equation}
\boxed{ \norm{|x|^{-\tau_j-1}u}_{L^{3,\frac32}(\R^3)}\leq C\norm{|x|^{-\tau_j-1}\alpha\cdot\nabla u}_{L^{\frac32}(\R^3)}}
 \label{eq:Carleman}
\end{equation}
for an appropriate sequence $\tau_j\to\ii$ and a universal constant $C$. We first discuss the proof of~\eqref{eq:Carleman} which we deduce from the similar inequality in~\cite[Eq.~(2.1)]{Seo-15}. By duality and density of $C^\ii_c(\R^3\setminus\{0\},\C^4)$ in $L^{\frac32,3}(\R^3)$ we can rephrase~\eqref{eq:Carleman} in the manner
$$\left|\int_{\R^3}\frac{g(x)^*\, u(x)}{|x|^{\tau_j+1}}\,dx\right|
\leq C\norm{g}_{L^{\frac32,3}(\R^3)}\norm{|x|^{-\tau_j-1}\alpha\cdot\nabla u}_{L^{\frac32}(\R^3)} $$
for all $u,g\in C^\ii_c(\R^3\setminus\{0\},\C^4)$. Here $g(x)^*$ denotes the transposition and complex conjugation of the vector $g(x)\in\C^4$. Letting $f=-i|x|^{-\tau_j-1}\alpha\cdot \nabla\, u$ which also belongs to $C^\ii_c(\R^3\setminus\{0\})$, this reduces the problem to showing that 
\begin{multline}
\left|\pscal{|x|^{-\tau_j-1}g,(\alpha\cdot p)^{-1}|x|^{\tau_j+1}f}_{L^2}\right|\leq C\norm{g}_{L^{\frac32,3}(\R^3)}\norm{f}_{L^{\frac32}(\R^3)},\\
\forall f,g\in C^\ii_c(\R^3\setminus\{0\}).
 \label{eq:Carleman2}
\end{multline}
Note that the two functions $|x|^{-\tau_j-1}g$ and $|x|^{\tau_j+1}f$ belong to $C^\ii_c(\R^3)$ and that $\1_B(\alpha\cdot p)^{-1}\1_B$ is a bounded operator on $L^2$ for any ball $B$, hence the left side of~\eqref{eq:Carleman2} makes sense for every $f,g\in C^\ii_c(\R^3\setminus\{0\})$. We estimate the kernel pointwise and obtain 
\begin{align*}
&\left|\pscal{|x|^{-\tau_j-1}g,(\alpha\cdot p)^{-1}|x|^{\tau_j+1}f}_{L^2}\right|\\
&\qquad=\frac{1}{4\pi}\left|\iint_{\R^3\times\R^3} \frac{g(x)^*\,\alpha\cdot(x-y)f(y)}{|x|^{\tau_j+1}|x-y|^3}|y|^{\tau_j+1}\,dx\,dy\right|\\
&\qquad\leq \frac{1}{4\pi}\iint_{\R^3\times\R^3} \frac{|g(x)|\,|f(y)|}{|x|^{\tau_j+1}|x-y|^2}|y|^{\tau_j+1}\,dx\,dy\\
&\qquad=\pscal{|x|^{-\tau_j-1}|g|,(-\Delta)^{-\frac12}|x|^{\tau_j+1}|f|}_{L^2(\R^3)}.
\end{align*}
The right side was studied in~\cite{Seo-15} where it is shown that 
$$\left|\pscal{|x|^{-\tau_j-1}g,(-\Delta)^{-\frac12}|x|^{\tau_j+1}f}_{L^2(\R^3)}\right|\leq C\norm{g}_{L^{\frac32,3}(\R^3)}\norm{f}_{L^{\frac32}(\R^3)}.$$
This concludes the proof of~\eqref{eq:Carleman2}. 

The way to deduce the result using~\eqref{eq:Carleman} is classical and works as in~\cite{AmrBerGeo-81,JerKen-85,Seo-15}. We quickly outline the argument for the convenience of the reader. Let $\chi(x):=\max(0,\min(2-|x|,1))$ which localizes in a neighborhood of the ball $B_1$ and set as usual $\chi_r(x):=\chi(x/r)$. Let also $\eta(x):=\max(0,\min(2,|x|-1))$ which localizes outside of the ball $B_1$ and $\eta_k(x):=\eta(2^kx)$. We consider $\Psi$ as in the statement and assume, after an appropriate space translation, that $\Psi$ vanishes to all orders at the origin. Using~\eqref{eq:Carleman} we estimate 
\begin{align*}
&\norm{|x|^{-\tau_j-1}\eta_k\chi_r \Psi}_{L^{3,\frac32}(\R^3)}\\
&\qquad\leq C\norm{|x|^{-\tau_j-1}\alpha\cdot\nabla(\eta_k\chi_r \Psi)}_{L^{\frac32}(\R^3)}\\
&\qquad\leq C\norm{|x|^{-\tau_j-1}\eta_k\chi_r V\Psi}_{L^{\frac32}(\R^3)}+C\norm{|x|^{-\tau_j-1}|\nabla\chi_r| \Psi}_{L^{\frac32}(B_{2r}\setminus B _r)}\\
&\qquad\qquad +C\norm{|x|^{-\tau_j-1}\chi_r|\nabla\eta_k| \Psi}_{L^{\frac32}(B_{2r}\setminus B _r)}\\
&\qquad\leq C'\norm{V}_{L^{3,\ii}(B_{2r})}\norm{|x|^{-\tau_j-1}\eta_k\chi_r\Psi}_{L^{3,\frac32}(\R^3)}+\frac{C\norm{\nabla\chi}_{L^3}}{r^{\tau_j+1}}\norm{\Psi}_{L^{3}(B_{2r}\setminus B_r)}\\
&\qquad\qquad +C2^{k(\tau_j+2)}\norm{\Psi}_{L^{\frac32}(B_{2^{1-k}})},
\end{align*}
where $C'$ equals $C$ multiplied by the constant in Hölder's inequality for Lorentz spaces. 
We then let $\eps_0:=1/(2C')$ in order to be able to put the first term on the left side and invert it. Namely, the condition is that 
$$\limsup_{r\to0}\norm{V}_{L^{3,\ii}(B_{2r})}\leq \frac1{2C'}.$$
Choosing $r$ small enough (depending only on $V$ and on the considered origin but not on $\Psi$ and $\tau_j$) we obtain
\begin{align}
\norm{\left(\frac{r}{|x|}\right)^{\tau_j+1}\eta_k\Psi}_{L^{3,\frac32}(B_r)}&\leq \norm{\left(\frac{r}{|x|}\right)^{\tau_j+1}\eta_k\chi_r \Psi}_{L^{3,\frac32}(\R^3)}\nn\\
&\leq \frac{C\norm{\nabla\chi}_{L^3}}{1-C'\norm{V}_{L^{3,\ii}(B_{2r})}}\norm{\Psi}_{L^{3}(B_{2r}\setminus B_r)}\nn\\
&\qquad +\frac{C2^{k(\tau_j+2)}\norm{\Psi}_{L^{\frac32}(B_{2^{1-k}})}}{1-C'\norm{V}_{L^{3,\ii}(B_{2r})}}.\label{eq:estim_tronquee}
\end{align}
We have 
$$2^{k(\tau_j+2)}\norm{\Psi}_{L^{\frac32}(B_{2^{1-k}})}\leq\left(2^{2k(\tau_j+2)+1-k}\int_{B_{2^{1-k}}}|\Psi|^2\right)^{\frac12}$$
which tends to 0 when $k\to0$, since $\Psi$ vanishes to all orders at the origin by assumption. Passing to the limit $k\to\ii$ in~\eqref{eq:estim_tronquee} we find
$$\norm{\left(\frac{r}{|x|}\right)^{\tau_j+1}\Psi}_{L^{3,\frac32}(B_r)}\leq \frac{C\norm{\nabla\chi}_{L^3}}{1-C'\norm{V}_{L^{3,\ii}(B_{2r})}}\norm{\Psi}_{L^{3}(B_{2r}\setminus B_r)}.$$
Taking then $\tau_j\to\ii$ gives $\Psi\equiv0$ on $B_r$. Iterating the argument gives $\Psi\equiv0$ on the whole connected domain $\Omega$. 
\end{proof}

Theorem~\ref{thm:unique_continuation} has the following immediate consequence.

\begin{cor}[Unique continuation for Dirac operators with general charge distributions]\label{cor:unique_continuation_general_charge}
Let $\mu$ be any finite signed Borel measure on $\R^3$, such that 
$$|\mu(\{R\})|<1\qquad \text{for all $R\in\R^3$.}$$ 
Then the eigenfunctions of the self-adjoint operator
$$D_0-\mu\ast\frac{1}{|x|}$$
defined in~\cite[Thm.~1]{EstLewSer-20a_ppt} can never vanish on a set of positive measure in $\R^3$. They can also not vanish to all orders at a point $x_1$ such that $\mu(\{x_1\})\leq \min(1/4,\eps_0/4)$ where $\eps_0$ is the constant from Theorem~\ref{thm:unique_continuation}.
\end{cor}

\begin{proof}[Proof of Corollary~\ref{cor:unique_continuation_general_charge}]
Let us denote by $R_1,...,R_K$ all the points for which $|\mu(\{R_j\})|\geq \min(1/4,\eps_0/4)$ where $\eps_0$ is the universal constant from Theorem~\ref{thm:unique_continuation} and set $\Omega:=\R^3\setminus\{R_1,...,R_K\}$. Then, by~\eqref{eq:domain_in_H1_loc} we know that an eigenfunction $\Psi$ is necessarily in $H^1_{\rm loc}(\Omega)$ so that $\nabla \Psi\in L^2_{\rm loc}(\Omega)$ and $\Psi\in L^6_{\rm loc}(\Omega)\subset L_{\rm loc}^{3,\frac32}(\Omega)$. In addition, with $V_\mu:=\mu\ast|x|^{-1}$ we can write  
\begin{equation}
|V_\mu|\leq |V_{\mu\1_{B_{\delta}(x_0)}}|+\frac{|\mu|(\R^3)}{\delta-r}\qquad\text{a.e.~in the ball $B_r(x_0)$.}
 \label{eq:decomp_weak_L3}
\end{equation}
The first potential satisfies\footnote{In our convention $\norm{|x|^{-1}}_{L^{3,\ii}}=1$.}
$$\norm{V_{\mu\1_{B_{\delta}(x_0)}}}_{L^{3,\ii}(B_r(x_0))}\leq |\mu|\big(B_{\delta}(x_0)\big)$$
which is less than $\eps_0/2$ for $\delta$ small enough (depending on $x_0$). For this fixed $\delta$ the last term in~\eqref{eq:decomp_weak_L3} is in $L^\ii$ and hence converge to 0 in $L^{3,\ii}(B_r(x_0))$ when $r\to0$. The same applies to the mass term $\beta$ in the Dirac operator $D_0=\alpha\cdot (-i\nabla)+\beta$. Thus
$$\limsup_{r\to0}\norm{-V_{\mu}+\beta}_{L^{3,\ii}(B_r(x_0))}\leq \frac{\eps_0}2$$
for every $x_0\in\Omega$. We are therefore exactly in the setting of Theorem~\ref{thm:unique_continuation}. If $\Psi$ vanishes to all orders at a point in $\Omega$, we deduce immediately that $\Psi\equiv0$. 

In case that $\Psi$ only vanishes on a set $A$ of positive measure, we argue like in~\cite{FigGos-92} and in the proof of Theorem~\ref{thm:existence_optimal_measure}, to deduce the existence of a point in $\Omega$ where $\Psi$ vanishes to all orders. We quickly outline the argument for the convenience of the reader. We pick a point $x_1$ of density of $\Omega\cap A$ and assume again that $x_1=0$ without loss of generality. Denote
$$\eps(r):=\frac{|B_r\setminus K|}{|B_r|}$$
which tends to 0 when $r\to0$. Let $\delta\leq1$ be so that $\mu(B_\delta)\leq 3/8$. Let $\chi_r$ be the same function as in the proof of Theorem~\ref{thm:unique_continuation}, which localizes around the origin. Then we have for $r\leq \delta/4$ 
\begin{align*}
\norm{\Psi}_{L^2(B_r)}=\norm{\Psi}_{L^2(B_r\setminus A)}&\leq |B_r\setminus A|^{\frac{1}{3}}\norm{\Psi}_{L^6(B_r\setminus A)}\\
&\leq |B_1|^{\frac13}r\eps(r)^{\frac13}\norm{\chi_r\Psi}_{L^6(\R^3)}\\
&\leq C_{\rm Sob}|B_1|^{\frac13}r\eps(r)^{\frac13}\norm{D_0\chi_r\Psi}_{L^2(\R^3)}\\
&\leq Cr\eps(r)^{\frac13}\norm{(D_0-V_{\mu\1_{B_{\delta}}})\chi_r\Psi}_{L^2}.
\end{align*}
Here we have used the Sobolev inequality and, in the last estimate, the fact that 
$$\norm{D_0\left(D_0-V_{\mu\1_{B_{\delta}}}\right)^{-1}}\leq C$$
for some universal constant $C$ for $\mu(B_{\delta})\leq 3/8$, by the Rellich-Kato theorem.  Using then the eigenvalue equation for $\Psi$, we obtain
\begin{align*}
\norm{(D_0-V_{\mu\1_{B_{\delta}}})\chi_r\Psi}_{L^2}&\leq  |\lambda|\norm{\chi_r\Psi}_{L^2}+\norm{V_{\mu\1_{\R^3\setminus B_{\delta}}}\chi_r\Psi}_{L^2}+\norm{|\nabla\chi_r|\Psi}_{L^2}\\
&\leq  \left(1+\frac{2|\mu|(\R^3)}{\delta}+\frac{1}{r}\right)\norm{\Psi}_{L^2(B_{2r})}.
\end{align*}
This gives
\begin{equation*}
\norm{\Psi}_{L^2(B_r)}\leq 2C\eps(r)^{\frac13}\left(1+|\mu|(\R^3)\right)\norm{\Psi}_{L^2(B_{2r})}
\end{equation*}
for a universal constant $C$. The rest of the argument goes like in~\cite[Section~3]{FigGos-92}.
\end{proof}


\end{document}